\documentclass[a4paper,11pt]{article}

\usepackage[latin1]{inputenc}				
\usepackage[english]{babel}					
\usepackage{indentfirst}						
\usepackage{color}								
\usepackage{verbatim}							

\usepackage[top=3.2cm, bottom=4.4cm, left=2.5cm, right=2.5cm]{geometry}		

\usepackage{amsmath}
\usepackage{amsthm}							
\usepackage{amsfonts}							
\usepackage{amssymb}							

\newtheorem{definition}{Definition}
\newtheorem{theorem}[definition]{Theorem}
\newtheorem{proposition}[definition]{Proposition}

\theoremstyle{definition}
\newtheorem{remark}[definition]{Remark}

\author{	Arnaud LIONNET \\ 
				\small  Oxford-Man Institute and Mathematical Institute \\
		         \small  University of Oxford, UK\\ 
    		 }
\title{Some results on general quadratic reflected BSDEs driven by a continous martingale}
\date{} 

\newcommand{\N}{\mathbb{N}}																				
\newcommand{\R}{\mathbb{R}}																					
\newcommand{\F}{\mathcal{F}}																					
\newcommand{\B}{\mathcal{B}}																					
\renewcommand{\H}{\mathcal{H}}																				
\renewcommand{\S}{\mathcal{S}}																				
\newcommand{\Exp}{\mathcal{E}}																				
\newcommand{\abs}[1]{\lvert#1\rvert }																		
\newcommand{\norm}[1]{\lVert#1\rVert}																	
\newcommand{\Norm}[1]{\left\lVert#1\right\rVert}													
\newcommand{\half}{\frac{1}{2}}																				
\newcommand{\bracket}[1]{\langle #1 \rangle}															
\DeclareMathOperator*{\esssup}{ess\, sup}																
\newcommand{\Et}[1]{E\bigg( #1 \Big| \F_t \bigg)}													

\begin{document}
\maketitle
\abstract{
We study the well-posedness of general reflected BSDEs driven by a continuous martingale, when the coefficient $f$ of the driver has at most quadratic growth in the control variable $Z$, with a bounded terminal condition and a lower obstacle which is bounded above. We obtain the basic results in this setting : comparison and uniqueness, existence, stability. 
For the comparison theorem and the special comparison theorem for reflected BSDEs (which allows one to compare the increasing processes of two solutions), we give intrinsic proofs which do not rely on the comparison theorem for standard BSDEs. This allows to obtain the special comparison theorem under minimal assumptions.
We obtain existence by using the fixed point theorem and then a series of perturbations, first in the case where $f$ is Lipschitz in the primary variable $Y$, and then in the case where $f$ can have slightly-superlinear growth and the case where $f$ is monotonous in $Y$ with arbitrary growth. 
We also obtain a local Lipschitz estimate in $BMO$ for the martingale part of the solution.
}



\section{Introduction}		

Since Pardoux and Peng initiated the systematic study of (non-linear) backward stochastic differential equations (BSDEs thereafter) in \cite{PardouxAndPeng}, these equations have been proved useful for a number of areas. They are intimately linked with the stochastic version of the maximum principle for stochastic optimal control problems, they provide probabilistic representations for the solutions to some partial differential equations (PDEs thereafter), and they are natural in finance where, beyond coming in through the use of the above fields, they provide a natural language to express the replication of a European option. 
Reflected BSDEs (RBSDEs) have applications in the same areas, where they are linked to optimal stopping, PDEs with obstacle (variational inequalities) and American options.
Consequently, a significant effort has been directed to studying the conditions under which these equations are well-posed.

``Reflected'' BSDEs are in fact BSDEs subject to a constraint : the solution process $Y$ is required to remain above a lower obtacle $L$. In order to achieve this, it is necessary to add to the usual dynamics $dY_s=-f_sds + Z_sdW_s$ a ``force'' $dK$ that drives $Y$ upward. One wants that extra term to be minimal, so that $K$ is only active to prevent $Y$ from passing below the obstacle $L$. This optimality condition (known as the Skorohod condition) is often expressed as $\int_0^T 1_{\{Y_s>L_s\}}dK_s =0$. So a reflected BSDE takes the following form :
	\begin{align} \label{equation-RBSDE.in.Brownian.setting}
		\left\{ \begin{aligned}
			dY_s &= - f(s,Y_s,Z_s)ds - dK_s + Z_s dW_s , \\
			Y_T &=\xi , \\
			Y_t &\ge L_t \ \text{for all $t \in [0,T]$} , \\
			K &\ \text{is continuous, increasing, starts from $0$ and}\ \int_0^T 1_{\{Y_s>L_s\}}dK_s=0	,
		\end{aligned} \right.		
	\end{align}
where the solution to be determined is now the triple $(Y,Z,K)$.

\paragraph*{}
Reflected BSDEs were introduced by El Karoui, Kapoudjian, Pardoux, Peng and Quenez in \cite{EKnKnPnPnQ}. These authors considered the case where $f$ is Lipschitz, the terminal condition is square-integrable and the lower obstacle a continuous square-integrable semimartingale, the natural extension of Pardoux and Peng \cite{PardouxAndPeng}.
The first results for quadratic BSDEs (that is, when $f$ is allowed to have a quadratic growth in $z$) were obtained by Kobylanski in \cite{Kobylanski}, under the assumption that $\xi$ is bounded and $f$ Lipschitz in $y$. Lepeltier and San Martin \cite{LepeltierandSanMartin98} allowed $f$ to have slightly superlinear growth in $y$. Kobylanski, Lepeltier, Quenez and Torrès \cite{KobylanskiAndLepeltierAndQuenezAndTorres} were then able to prove the analogue results for RBSDEs by extending the techniques of \cite{Kobylanski} and \cite{LepeltierandSanMartin98}. 
The coefficient $f$ can have any growth in the $y$ variable if it satisfies the monotonicity condition (see for instance Pardoux \cite{Pardoux}), which is encountered in reaction-diffusion equations. Briand, Lepeltier and San Martin \cite{BriandAndLepeltierAndSanMartin} studied the case of quadratic growth BSDEs with such an assumption, and this was then extended to reflected BSDEs by Xu \cite{Xu}.
Briand and Hu \cite{BriandAndHu2006,BriandAndHu2008} extended Kobylanski's results \cite{Kobylanski} to the case of an unbounded terminal condition, and this case was further studied in the recent works of Delbaen, Hu and Richou \cite{DelbaenAndHuAndRichou2011, DelbaenAndHuAndRichou2013} and Barrieu and El Karoui \cite{BarrieuAndElKaroui}.  
Lepeltier and Xu \cite{LepeltierAndXu} could then treat the case of RBSDEs with unbounded $\xi$, while Bayraktar and Yao \cite{BayraktarAndYao} removed the condition that $L$ be bounded. 
Work has also been done regarding the regularity of the obstacle $L$, for instance Peng and Xu \cite{PengAndXu} worked with $L^2$ obstacles.
However, the case of quadratic BSDEs remains significantly more difficult than that of Lipschitz BSDEs, and the methods used initially are often quite involved. 
Recently, Tevzadze \cite{Tevzadze} and Briand and Elie \cite{BriandAndElie} gave simpler approaches for the case when $\xi$ is bounded.

The above works concerning the well-posedness of reflected BSDEs considered a Brownian setting. However, BSDEs have been studied in a general martingale setting (see El Karoui and Huang \cite{ElKarouiAndHuang}, Tevzadze \cite{Tevzadze}, Morlais \cite{Morlais}, Barrieu and El Karoui \cite{BarrieuAndElKaroui}), and in a general filtered probability space in Cohen and Elliott \cite{CohenAndElliott}. 
In this paper, we obtain the well-posedness of a general class of quadratic RBSDEs driven by a continuous martingale and with a bounded terminal condition in a simple, self-contained way. We show the existence of solutions in the cases where the dependence of $f$ in $y$ is Lipschitz, slightly-superlinear or monotone with arbitrary growth. We also obtain the special comparison theorem for the increasing processes under minimal assumptions. Finally, we obtain a local Lipschitz estimate in BMO for the martingale part of the solution.

\paragraph*{}
In section \ref{section-comparisontheorem}, 
we first obtain the standard comparison theorem in our setting, using a linearization and the BMO argument from Hu, Imkeller and Müller \cite{HuAndImkellerAndMuller}, as opposed to via an optimal stopping representation and comparison for BSDEs (see \cite{KobylanskiAndLepeltierAndQuenezAndTorres}). 
We note that this result, which guarantees uniqueness, holds naturally for $f$ only locally Lipschitz in $y$, instead of globally Lipschitz as often assumed (\cite{BriandAndElie}, \cite{Tevzadze}).
We then prove the special comparison theorem for reflected BSDEs, which allows one to compare the increasing processes when one RBSDE solution dominates another. 
This theorem was first proved in Hamadène, Lepeltier and Matoussi \cite{HamadeneAndLepeltierAndMatoussi}, and reused in Peng and Xu \cite{PengAndXu}. In the papers (\cite{HamadeneAndLepeltierAndMatoussi}, \cite{PengAndXu}, \cite{LepeltierAndMatoussiAndXu}, \cite{KobylanskiAndLepeltierAndQuenezAndTorres}) where it appears, the proof always relies on the penalization approach to reflected BSDEs and the comparison theorem for standard BSDEs, comparing quantities which, at the limit, become the increasing processes. 
The statement and the new proof we provide here are more intimately related to the nonlinear Snell envelope approach reflected BSDEs and hold under minimal assumptions. In particular, because we don't rely on a comparison theorem, they hold without the regularity assumptions usually made on $f$. 

In section \ref{section-existenceandstability} we prove the existence of solutions to the reflected BSDEs when $f$ is quadratic in $z$ an Lipchitz in $y$. To this end, we generalize the technique introduced by Tevzadze \cite{Tevzadze} for BSDEs.
The idea there is to first use the fixed point theorem to obtain a solution to a quadratic BSDEs when $f(\cdot,0,0)$ and $\xi$ are sufficiently small ($f(s,0,0)ds$ is the residual drift, that drives the solution even if $(Y_s,Z_s)=(0,0)$), and then to build upon this partial result to obtain a solution for general $f(\cdot,0,0)$ and $\xi$.

This technique can be understood as a type of ``vertical'' splitting and recombination, and is in that sense an analogue to what is done for Lipschitz BSDEs. In that classical case, if one works with the natural norm on the space where one looks for solutions, which in that context is the space of square-integrable processes, one finds that the fixed point theorem applies if the time interval is small enough. A natural way to use this is then to split (``horizontally'') a general time interval into pieces small enough that one can obtain a solution on each interval, and patch them together to obtain a solution on the whole interval. For quadratic BSDEs, since one can apparently solve the BSDE only for small data, the idea is to split a general set of data into pieces small enough that one can obtain a solution for each piece, and then combine them to obtain a solution to the initial problem (see \cite{Touzi} and \cite{KaziTaniAndPossamaiAndZhou}). One can also understand this method as a series of perturbations. One first solves a BSDE with microscopic data, then successively solves pertubation equations and adds the associated solutions, allowing the size of the data to grow at each step. At the end, one has built a solution to the initial BSDE with macroscopic data. 

In order to use the fixed point theorem, one mainly needs to understand the underlying backward stochastic problem. For BSDEs, this underlying problem is the semimartingale decomposition. For reflected BSDEs, it is a Snell envelope problem. However, for the perturbation procedure to work well, the underlying problem should be a linear problem (for instance, it has been applied recently in Kazi-Tani, Possamai and Zhou \cite{KaziTaniAndPossamaiAndZhou} to BSDEs with jumps). This way, the equations satisfied by the perturbations are of the same nature as the equations satisfied by the solutions. This is not the case for reflected BSDEs.
It is however possible to identify the equation that a perturbation should satisfy. The obstacle cannot be perturbed during the procedure, but this can be dealt with by assuming from the start that it is negative, a case which covers all the others by a simple translation. In particular, unlike in \cite{KobylanskiAndLepeltierAndQuenezAndTorres}, we don't need $L$ to be bounded but only require it to be upper bounded.

We then study the stability of the solution with respect to changes in the terminal condition $\xi$ and in the residual drift $f(\cdot,0,0)$. 
We obtain for the martingale part of the solution a local Lipschitz estimate in the space $BMO$. Global Lipschitz bounds in $\H^p$ were obtained already in Briand, Delyon, Hu, Pardoux and Stoica \cite{BriandAndDelyonAndHuAndPardouxAndStoica} (see also Briand and Confortola \cite{BriandAndConfortola}, Ankirchner, Imkeller and Dos Reis \cite{AnkirchnerAndImkellerAndDosReis}). Kazi-Tani, Possamai and Zhou \cite{KaziTaniAndPossamaiAndZhou} provide a global $\half$-Hölder estimate in the smaller space $BMO$. Here, we can obtain a stronger regularity for small perturbations, essentially by bootstrapping a weaker regularity result.

Finally, in section \ref{section-generalisation}, we extend the scope of the existence theorem of section \ref{section-existenceandstability}. 
In that latter case, $f$ is Lipschitz in $y$ and the sequence of perturbations described above can be performed uniformly without problem. However, when the first derivative $f_y$ is not a bounded function of $y$, the maximal allowed size for a perturbation depends on the size of the solution to the reflected BSDE that one wants to perturb, so it is not clear \textit{a priori} that the procedure would terminate after finitely many perturbations. We show, however, that this is the case as soon as one can obtain an \textit{a priori} bound for $Y$ in $\S^\infty$.
We can therefore extend the existence theorem of section \ref{section-existenceandstability} to the case where $f$ is slightly-superlinear and to the case where $f$ is monotone with arbitrary growth (as studied respectively by \cite{KobylanskiAndLepeltierAndQuenezAndTorres} and \cite{Xu} in a Brownian setting), using the same perturbation technique.

In the following section, we motivate the general continuous-martingale setting, and specify the notation and the framework that will be used throughout the paper.

  
\section{Setting}					\label{section-setting}

\paragraph{General, continuous-martingale setting.}\ \\
The reflected BSDE (\ref{equation-RBSDE.in.Brownian.setting}) is set in a Brownian setting. Even in this setting, it is sometimes useful to consider that the reference increasing process is not the time $ds$. For instance, Pardoux and Zhang observed in \cite{PardouxAndZhang} that when looking at the BSDE associated with a semilinear parabolic PDE in a domain with Neumann boundary conditions (so-called ``generalized BSDEs''), the drift term is of the form $f(Y_s,Z_s)ds + f'(Y_s)dA_s$ where $A$ is the local time of the underlying diffusion on the regular boundary, and is therefore orthogonal to the Lebesgue measure. In that case, one can enhance the increasing process by setting $dC = ds + dA$, and find $f''$ such that $f(Y_s,Z_s)ds + f'(Y_s)dA_s = f''(Y_s,Z_s)dC$ (see El Karoui and Huang \cite{ElKarouiAndHuang}). 

The reference martingale $M$ need not be a Brownian motion, and in particular need not enjoy the martingale representation property (see El Karoui, Peng and Quenez \cite{ElKarouiAndPengAndQuenez}). The martingale part $N$ of the solution then has the decompotion $N = \int ZdM + N^\perp$ on the reference martingale $M$, with $N^\perp$ orthogonal to $M$ (i.e. $\bracket{N^\perp , M}=0$). The quadratic variation $d\langle M \rangle$ of $M$ is assumed to be absolutely continuous (component-wise) with respect to $dC$ (one can always enhance $dC$ so that it becomes the case). Write then $d\langle M \rangle = a dC = \sigma \sigma^* dC$. Since $Z_s$ is uniquely determined only when no component of $d\langle M \rangle_s$ is zero, the drift term will only depend on $Z\sigma$. 
The drift is also allowed to depend on $N^\perp$ through a quadratic term $g_s d\langle N^\perp \rangle_s$ and a term $d\langle \nu,N^\perp \rangle_s$ linear in $N^\perp$.

So in the end, we will study the following general reflected BSDE :
	\begin{align} \label{equation-RBSDE.of.reference}
		\left\{ \begin{aligned}
			dY_s &= - dV(Y,N)_s - dK_s + dN_s ,	\\
			Y_T &=\xi  , 									\\
			Y_t &\ge L_t \text{ for all } t \le T	\text{, and}	\\
			K &\text{ increasing, continuous, starting from $0$ and such that } 1_{\{Y_s > L_s\}} dK_s=0	
		\end{aligned} \right.
	\end{align}
where the drift is given by
	\begin{align*}
		dV(Y,N)_s=f(s,Y_s,Z_s \sigma_s)dC_s + d\langle \nu, N^\perp \rangle_s + g_s d\langle N^\perp \rangle_s .
	\end{align*}
This is referred to as the reflected BSDE of data $(V,\xi,L)=(f,\nu,g,\xi,L)$.

\paragraph*{}
The framework is a filtered probability space $\big( \Omega, \F=(\F_t)_{t \in [0,\mathbb{T}]},P \big)$ satisfying the usual conditions, where $\mathbb{T} > 0$ is a finite time horizon. $T$ is an $\F$-stopping time valued in $[0,\mathbb{T}]$ (bounded stopping time).
The continuous square-integrable martingale $M$ is assumed to be $BMO$ (see below). All the processes considered are continuous.

$C$ is a continuous and progressively measurable increasing process (starting from $0$) such that, roughly, all the finite variational processes which are related to the data (not depending on the solution) are absolutely continuous with respect to it. In particular, $d\langle M \rangle_s = a_s dC_s = \sigma_s \sigma_s^* dC_s$. It is assumed that the positive symmetric matrix $a$ (or equivalently $\sigma$) is bounded away from 0 and infinity (i.e. bounded and uniformly elliptic). 

The data of the BSDE (coefficients $f,\nu,g$ of the drift $V$, terminal condition $\xi$, obstacle $L$) are as follows :
	\begin{itemize}
		\item $f : \Omega \times [0,T] \times \R \times M_{1,d}(\R) \rightarrow \R$ is $Prog \otimes \B(\R) \otimes \B(M_{1,d}(\R))$-measurable, where $M_{n,d}(\R)$ is the space of $n \times d$ matrices with entries in $\R$. $Prog = Prog(\F_T)$ is the progressively measurable sigma-field on the interval $[|0,T|]$ (the set of pairs $(\omega,t)$ such that $t \le T(\omega)$).
		\item $\nu$ is a BMO martingale orthogonal to $M$ (that is $\langle \nu,M \rangle=0$)
		\item $g$ is a progressively measurable and bounded scalar process.
		\item $\xi$ is an $\F_T$-measurable, bounded random variable.
		\item $L$ is a continuous semimartingale bounded above.
	\end{itemize}
Throughout the paper, we assume that $f$ has at most quadratic growth in the variable $z$, in the following sense :
	\begin{itemize}
		\item[($\mathbf{A_{qg}}$)] There exists a growth function $\lambda(\cdot)$ (i.e. $\lambda : \R \rightarrow \R_+$ symmetric, increasing on $\R_+$, bounded below by 1) and a positive process $h \in L^2_{BMO}$ (i.e. $\int h dM \in BMO$, see below) such that :
			\begin{align*}
				\abs{f(t,y,z)} \le \lambda(y) \big(h_t^2 + \abs{z}^2 \big) .
			\end{align*}
	\end{itemize}
The assumption as written above allows for any growth in $y$, although more specific assumptions on this are made in sections \ref{section-comparisontheorem}, \ref{section-existenceandstability} and \ref{section-generalisation}.
 
\paragraph{Solutions to the reflected BSDE.}\ \\
A solution to the reflected BSDE is generally understood as a triple $S=(Y,N,K)$ where Y is a semimartingale, N a square-integrable martingale ($\in \H^2$) and K an increasing process (starting from $0$), such that (\ref{equation-RBSDE.of.reference}) is satisfied, with $N=\int ZdM + N^\perp$.

Note that a solution can also be understood as a pair $S=(Y,N)$ such that, definining $K$ from $K_0=0$ and the dynamics equation in (\ref{equation-RBSDE.of.reference}), K is indeed found to be increasing and satisfies the Skorohod condition. This will often be what is meant by solution in the rest of the paper. 

Under the assumption of quadratic growth and bounded terminal condition, we consider only bounded solutions : $Y \in \S^\infty$. For those, $N$ is found to be a BMO martingale (see the a priori estimate of proposition \ref{proposition-aprioriestimate} below). So a solution will always be understood as being in $\S^\infty \times BMO$ ($\times \mathcal{A}$), where BMO and $\mathcal{A}$ are described below.

\paragraph{Spaces of processes, notation.}\ \\
We will make use of the following particular spaces.
	\begin{itemize}
		\item $BMO(P)$ is the space of all the BMO $P$-martingales, that is those for which the norm 
				\begin{align*}
					\norm{N}^2_{BMO(P)} = \sup_{t \in \mathcal{T}_0^T} \norm{E_P( \langle N \rangle_T - \langle N \rangle_t | \F_t)}_\infty
				\end{align*}				
			 is finite, where $\mathcal{T}_0^T$ is the set of stopping times $t$ such that $0 \le t \le T$. The mention of the measure $P$ will be omitted whenever no confusion is possible. When $X \in BMO$, then $\mathcal{E}(X)$ is a UI martingale, so one can define a measure $Q$ by stating that on $\F_t$, $\frac{dQ}{dP}=\mathcal{E}(X)_t$. Also, we will use frequently the fact that for any $N \in BMO(P)$, $\widetilde{N}=N-\bracket{X,N}$ is in $BMO(Q)$ (cf Kazamaki \cite{Kazamaki}, theorems 2.3 and 3.3). 
		\item $L^2_{BMO}$ is the space of processes $h$ such that $\int h dM \in BMO$. We equip it with the norm
			\begin{align*}
				\norm{h}_{L^2_{BMO}}^2 = \sup_t \Norm{ \Et{ \int_t^T h_s^2 dC_s } }_\infty
			\end{align*}
		\item $\mathcal{A}$ is the space of accumulators, that is : progressively measurable, continuous, increasing processes starting from $0$.
		\item $L^{\infty,2}$ is the space of processes $x$ such that $\int_0^T \abs{x_s}^2 dC_s \in L^\infty$, and $L^{\infty,1}$ is that of processes such that $\int_0^T \abs{x_s} dC_s \in L^\infty$, with norms $\norm{x}_{\infty,2}^2 = \norm{\int_0^T \abs{x_s}^2 dC_s}_\infty$ and $\norm{x}_{\infty,1} = \norm{\int_0^T \abs{x_s} dC_s}_\infty$. \\
	\end{itemize}
In the growth assumptions on $y$ made later on, we use a fixed positive process $r \in L^{\infty,2}$, which is part of the framework (like $T$, $M$, $C$). 
It is there to take into account the fact that $C_T$ might not be bounded (in facts, if $C$ incorporates a local time of a diffusion on a regular boundary, it has exponential moments but is not bounded). 
In the case where $dC_s=ds$ and $T$ is a constant, $r=1$ and $\norm{r}_{\infty,2} = \sqrt{T}$.

For processes $S=(Y,N) \in \S^\infty \times BMO(Q)$, we use the norm $\norm{S}_Q^2 := \norm{Y}_{\S^\infty}^2 + \norm{N}_{BMO(Q)}^2$.
As usual the mention of the measure $Q$ is omitted whenever no confusion arises.
By $L^\infty$ we denote the space of bounded random variables, and we use the norm $\norm{\ \ }_\infty$, whether those random variables are $\R$-valued (like $\xi$) or path-valued (like $g$).

\paragraph{BMO property for $N$.}
	\begin{proposition}			\label{proposition-aprioriestimate}
		Let $f$ satisfy $(\mathbf{A_{qg}})$, $\nu \in BMO$ and $g$ be bounded. 
		Let $Y$ be a continuous semimartingale, $N$ be a square-integrable martingale and $K$ be an increasing process such that $Y$ has the decomposition :
			\begin{align*}
					dY &= -dV(Y,N) - dK + dN	 , 
			\end{align*}
		where $dV(Y,N)_s = f(s,Y_s,Z_s \sigma_s)dC_s + d\langle \nu,N^\perp \rangle_s + g(s)d\langle N^\perp \rangle_s$.
		If $Y$ is bounded (i.e. $Y \in \mathcal{S}^\infty$), then $N \in BMO$ and $K \in \mathcal{A}_{BMO}$.
	\end{proposition}

Here, $\mathcal{A}_{BMO}$ refers to the increasing processes $K \in \mathcal{A}$ such that the norm $\norm{K}_{\mathcal{A}_{BMO}} = \sup_t \norm{ E\big( K_T - K_t |\F_t \big) }_\infty$ is finite.
Note that this statement is, to some extent, not so much about solutions to a (possibly reflected) BSDE but about quadratic semimartingales (see Barrieu and El Karoui \cite{BarrieuAndElKaroui}), and quadratic semimartingales are considered here up to a monotonous process. 
		
	\begin{remark}			\label{remark-aprioriestimate}
		The result implies in particular the following :  if $Y$ is a bounded semimartingale with decomposition $dY=-dV-dK+dN$, with $K$ monotonous (which boils down to increasing, up to considering $-Y$) and if the process $V$ is in $L^1_{BMO}$ (i.e. $\sup_t \norm{ E\big( \int_t^T \abs{dV_s} \big|\F_t\big) }_\infty < +\infty$), then $N \in BMO$ and $K \in \mathcal{A}_{BMO}$.
	\end{remark}

	\begin{proof}
		The proof uses the usual exponential transform. Let $\mu \in \R$, whose sign and value will be chosen later. By Itô's formula for the process $\exp(\mu Y)$ between a stopping time $t \in \mathcal{T}_0^T$ and $T$ one has 
			\begin{equation} \label{equation-exptransform}
			\begin{split}
				e^{\mu Y_t} - \mu \int_t^T e^{\mu Y_s}dK_s + \frac{\mu^2}{2} \int_t^T e^{\mu Y_s} d\langle N \rangle_s &= e^{\mu Y_T} + \mu \int_t^T e^{\mu Y_s} dV_s 						\\
					&\qquad - \mu \int_t^T e^{\mu Y_s} dN_s \ .
			\end{split}
			\end{equation}
		Since $Y \in \mathcal{S}^\infty$, the process $e^{\mu Y}$ is bounded, and since $N$ is a square-integrable martingale, $\int e^{\mu Y} dN$ is a martingale. 
		We have 
			\begin{align*}
				\abs{dV_s} &\le \abs{f(s,Y_s,Z_s\sigma_s)}dC_s + \abs{d\langle \nu, N^{\perp} \rangle_s} + \abs{g_s}\abs{d\langle N^{\perp} \rangle_s} ,
			\end{align*}
		and using the quadratic growth assumption on $f$ we have 
			\begin{align*}
				\abs{f(s,Y_s,Z_s\sigma_s)} \le \lambda(Y_s) \big(h_s^2 + \abs{Z_s\sigma_s}^2\big) \le \Lambda \big(h_s^2 + \abs{Z_s\sigma_s}^2\big) ,
			\end{align*}	
		where $\Lambda = \lambda(\norm{Y}_{\S^\infty})$.
		Using the Kunita-Watanabe inequality and $ab \le a^2 + b^2$, we see that
			\begin{align*}
				E\bigg(  \int_t^T e^{\mu Y_s} \abs{dV_s} \Big| \F_t \bigg) &\le \Lambda \Et{ \int_t^T e^{\mu Y_s} \big( h_s^2 + \abs{Z_s \sigma_s}^2 \big) dC_s } 
						+ \Et{ \int_t^T e^{\mu Y_s} d\bracket{\nu}_s } 	\\
						&\qquad   \Et{ \int_t^T e^{\mu Y_s} d\bracket{N^\perp}_s } + \norm{g}_\infty \Et{ \int_t^T e^{\mu Y_s} d\bracket{N^\perp}_s } \ .
			\end{align*}
		Recall that by the orthogonality of $M$ and $N^\perp$, $d\bracket{N} = \abs{Z\sigma}^2dC + d\bracket{N^\perp}$, and therefore both $\abs{Z\sigma}^2dC$ and $d\bracket{N^\perp}$ are less than or equal to $d\bracket{N}$. Therefore, 
			\begin{align*}
				E\bigg(  \int_t^T e^{\mu Y_s} \abs{dV_s} \Big| \F_t \bigg) &\le \Lambda \Et{ \int_t^T e^{\mu Y_s} h_s^2 dC_s }	+ \Et{ \int_t^T e^{\mu Y_s} d\bracket{\nu}_s } 	\\
						&\qquad   \big(\Lambda + 1 + \norm{g}_{\infty}\big) \Et{ \int_t^T e^{\mu Y_s} d\bracket{N}_s }  \ .
			\end{align*}		
		So, setting $b=(\Lambda + 1 + \norm{g}_\infty)$, and taking the conditional expectation of (\ref{equation-exptransform}) with respect to $\F_t$, one has 
			\begin{align*}
				0 -& \mu \Et{ \int_t^T e^{\mu Y_s} dK_s } + \Big\{\frac{\mu^2}{2} - \abs{\mu} b \Big\} \Et{ \int_t^T e^{\mu Y_s} d\bracket{N}_s }  \\
					&\qquad\qquad \le e^{\abs{\mu}\norm{Y}_{\S^\infty}} + \abs{\mu}\bigg[ \Lambda \Et{ \int_t^T e^{\mu Y_s} h_s^2 dC_s } + \Et{ \int_t^T e^{\mu Y_s}  d\bracket{\nu}_s } \bigg] - 0 .
			\end{align*}
		We now choose $\mu = -4b$, so $\frac{\mu^2}{2} - \abs{\mu} b = 4b^2$. 
		Since $b \ge 1$, $b^2 \ge b$. 
		We now use the fact that $e^{-\abs{\mu}\norm{Y}_{\S^\infty}} \le e^{\mu Y_s} \le e^{\abs{\mu}\norm{Y}_{\S^\infty}}$ and take the the $\sup_t$, so we obtain finally
			\begin{align*}
				\norm{K}_{\mathcal{A}_{BMO}} + \norm{N}_{BMO}^2 \le \frac{e^{8b\norm{Y}_{\S^\infty}}}{2b} \Big[ 1 + 4b \big( \Lambda \norm{h}_{L^2_{BMO}}^2 + \norm{\nu}_{BMO}^2 \big) \Big]
						< +\infty	.
			\end{align*}
	\end{proof}

Note however that while it indeed gives a bound for $N \in BMO$, this estimate does not guarantee that if $\norm{Y}_{\S^\infty} \longrightarrow 0$ then $\norm{N}_{BMO} \longrightarrow 0$.

  
\section{Comparison theorems and uniqueness.}					\label{section-comparisontheorem}

\subsection{Comparison theorem.}

We prove below the comparison theorem in our setting, which guarantees uniqueness in the existence theorems of sections \ref{section-existenceandstability} and \ref{section-generalisation}.
The regularity assumption that we require for the theorem to hold is, for notational simplicity, the following :
	\begin{itemize}
		\item[($\mathbf{A_{Df}}$)] The function $f$ is of class $\mathcal{C}^1$ (in the variable $(y,z)$, for all $\omega,t$) with 
			\begin{align*}
				\abs{f_y(t,y,z)} \le \rho(y) r_t^2	\qquad\text{and}\qquad \abs{f_z(t,y,z)} \le \rho'(y) (h_t+\abs{z}) ,
			\end{align*}					
		for some growth functions $\rho$ and $\rho'$, and some positive process $h \in L^2_{BMO}$.
	\end{itemize}

	\begin{theorem}			\label{proposition-comparisontheorem}
		Consider two sets of data $(f,\nu,g,\xi,L)$ and $(f',\nu,g',\xi',L')$, and assume that :
			\begin{enumerate}
				\item there exist solutions $(Y,N,K)$ and $(Y',N',K')$ to the corresponding reflected BSDEs,
				\item the parameters are ordered : $f' \le f$, $g' \le g$, $\xi' \le \xi$ and $L' \le L$,
				\item $f$ is regular enough : it satisfies ($\mathbf{A_{Df}}$).
			\end{enumerate}
		Then one has $Y' \le Y$. 
	\end{theorem}

While the proof given in \cite{KobylanskiAndLepeltierAndQuenezAndTorres} in a Brownian setting uses an optimal stopping representation and the comparison theorem for BSDEs, 
we rely here on a classical linearization argument and the properties of solutions to a linear BSDE. 
More precisely, we study the positive part $(\Delta Y)^+$, where $\Delta X = X' - X$ for a generic quantity $X$, and show that $(\Delta Y)^+ \le 0$. 

	\begin{proof}
		Denoting by $l$ the local time of $\Delta Y$ in $0$, the Itô-Tanaka formula gives
			\begin{equation}	\label{equation-DeltaY+}
			\begin{split}
				d(\Delta Y)_s^+ &= 1_{\{ \Delta Y_s > 0 \}} d\Delta Y_s +\frac{1}{2}dl_s\\
				&= 1_{\{ \Delta Y_s > 0 \}} \Big[ - d\Delta V_s - d\Delta K_s + d\Delta N_s \Big] + \frac{1}{2}dl_s \ .
			\end{split}
			\end{equation}			
		Now,  gathering terms, rewriting differences, and linearizing some,
			\begin{align*}
				d\Delta V_s &= \Big[ f'(s,Y'_s,Z'_s\sigma_s)-f(s,Y_s,Z_s\sigma_s) \Big]dC_s 			\\
					&\qquad + d\langle \nu, (N')^\perp - N^\perp \rangle + g'_s d\langle (N')^\perp \rangle_s - g_sd\langle N^\perp \rangle_s			\\
				\\				
				&= \Big[ (\Delta f)(s,Y'_s,Z'_s\sigma_s) + f(s,Y'_s,Z'_s\sigma_s)-f(s,Y_s,Z_s\sigma_s) \Big]dC_s 		\\
					&\qquad +  d\langle \nu, \Delta N^\perp \rangle_s +  (\Delta g)_s\, d\langle (N')^\perp \rangle_s +	g_s \Big[ d\langle (N')^\perp \rangle - d\langle N^\perp \rangle_s \Big]	\\
				\\				
				&= \Big[ (\Delta f)(s,Y'_s,Z'_s\sigma_s) + f(s,Y'_s,Z'_s\sigma_s) - f(s,Y_s,Z'_s\sigma_s) + f(s,Y_s,Z'_s\sigma_s) - f(s,Y_s,Z_s\sigma_s) \Big]dC_s 		\\
					&\qquad +  d\langle \nu, \Delta N^\perp \rangle_s +  (\Delta g)_s\, d\langle (N')^\perp \rangle_s +	g_s  d\langle (N')^\perp + N^\perp , \Delta N^\perp \rangle \\
				\\
				&= \Big[ \Delta f + F_y \Delta Y + F_z \Delta Z \sigma \Big]dC 		\\
					&\qquad +  d\langle \nu', \Delta N^\perp \rangle +  (\Delta g)\, d\langle (N')^\perp \rangle	, \\
			\end{align*}		
		where 
			\begin{align*}
				F_y(s)&=F_y(s,Y_s,Y'_s,Z'_s\sigma_s)=\int_0^1 f_y(s,Y_s+u\Delta Y_s,Z'_s\sigma_s)du  \ , \\
				F_z(s)&=F_z(s,Y_s,Z_s\sigma_s,Z'_s\sigma_s)=\int_0^1 f_z(s,Y_s,Z_s\sigma_s + u\Delta Z_s\sigma_s)du \text{ and}\\
				\nu' &= \nu + \int gd(N^\perp+(N')^\perp)																		\ .
			\end{align*}			
		So we can rewrite (\ref{equation-DeltaY+}) as
			\begin{align} \label{equation-DeltaY+linear}
				d(\Delta Y)^+ &= -dD - F_y (\Delta Y)^+ dC + 1_{\{ \Delta Y > 0 \}} \Big[ d\Delta N - F_z \Delta Z\sigma dC -  d\langle \nu', \Delta N^\perp \rangle \Big] \ ,
			\end{align}
		where
			\begin{align*}
				dD = 1_{\{ \Delta Y > 0 \}} \Big[ \Delta f dC_s + \Delta g\, d\langle (N')^\perp \rangle + d\Delta K \Big] - \frac{1}{2}dl
			\end{align*}		
		is a decreasing process. Indeed $\Delta f \le 0$, $\Delta g \le 0$, $dl \ge 0$ and
			\begin{align*}
				1_{\{ \Delta Y > 0 \}} d\Delta K = \underbrace{ 1_{\{ \Delta Y > 0 \}} dK'}_{=0} - \underbrace{ 1_{\{ \Delta Y > 0 \}} dK}_{\ge 0} \le 0
			\end{align*}
		because $dK \ge 0$ and on $\{ \Delta Y > 0 \}$ we have $Y'>Y \ge L \ge L'$, hence $dK'=0$.
		
		$(\Delta Y)^+$ is therefore seen as the solution to a linear equation (\ref{equation-DeltaY+linear}). 
		Define the integrating factor $B_t=e^{\int_0^t F_y(u) dC_u}$ and the measure $Q$ by $\frac{dQ}{dP}=\Exp(\int F_z \sigma^{-1} dM + \nu')_t$. 
		By the assumption on $f_z$ and the fact that $h \in L^2_{BMO}$ and $\nu$, $N$, and $N'$ are in BMO (recall proposition \ref{proposition-aprioriestimate} and the definition of a solution), 
		$\int F_z \sigma^{-1} dM + \nu'$ is in BMO, and therefore Q is indeed well defined. 
		Then, $\widetilde{\Delta N} =  \Delta N - \int F_z \Delta Z\sigma dC_s -  \langle \nu', \Delta N^\perp \rangle $ is a $BMO(Q)$-martingale. 
		By the assumption on $f_y$, the process $B_\cdot$ is bounded so $\int B_u 1_{\Delta Y_u > 0} d\widetilde{\Delta N_u}$ is again a $Q$-martingale. 
		Therefore, looking at the dynamic of $\widehat{Y}=BY$ under $Q$ we finally find that
			\begin{equation*}
				0 \le (\Delta Y_t)^+ = E_Q \bigg( e^{\int_t^T F_y(v)dC_v} (\Delta \xi)^+ + \int_t^T e^{\int_t^u F_y(v)dC_v} dD_u \Big| \F_t \bigg) \le 0 \ .
			\end{equation*}	
	\end{proof}

	\begin{remark}
		The previous theorem is stated, for convenience, for a function $f$ which is $C^1$. Typically, for the comparison theorem one only requires $f$ to be locally Lipschitz, in which case the processes $F_y,F_z$ have to be replaced by the differential quotients :  $\delta_y f(Y,Y',Z'\sigma) = \frac{f(Y',Z')-f(Y,Z')}{Y'-Y} 1_{Y \neq Y'}$, etc, and the above proof works as long as $F_y \in L^{\infty,1}$ and $F_z \in L^2_{BMO}$. 
		These criteria are satisfied as soon as 
			\begin{itemize}
				\item[($\mathbf{A_{\text{lLip}}}$)] There exist growth functions $\rho$ and $\rho'$, and a process $h \in L^2_{BMO}$ such that 
					\begin{align*}
						\abs{f(t,y',z') - f(t,y,z)} \le \rho(y,y') r_t^2 \abs{\Delta y} +  \rho'(y,y') \big(h_t + \abs{z} + \abs{z'} \big) \abs{\Delta z} \ . 
					\end{align*}		
				\end{itemize}
			Note : when $\rho, \rho'$ are constants this is the standard assumption of local Lipchitz regularity made in the quadratic BSDE literature (see for instance Briand and Elie \cite{BriandAndElie} and Tevzadze \cite{Tevzadze}). However, since we are dealing with bounded solutions, the assumption can be weakened to the case ($\mathbf{A_{\text{lLip}}}$) where $\rho,\rho'$ are growth functions.
	\end{remark}

  
\subsection{Special comparison theorem.}

When the two sets of data are in a comparison configuration and when the lower obstacles are the same, one can say more than $Y' \le Y$ and also compare the increasing processes of the two solutions, $K'$ and $K$. 

	\begin{proposition}			\label{proposition-refinedcomparisontheorem}
		Let $(f,g,\nu,\xi,L)$ and $(f',g',\nu,\xi',L)$ be some data, and assume that :
			\begin{enumerate}
				\item there exist solutions $S=(Y,N,K)$ and $S'=(Y',N',K')$ to the corresponding RBSDEs,
				\item the drift coefficients are ordered : $f' \le f$, $g' \le g$, 
				\item $Y'$ is dominated by $Y$ : $Y' \le Y$.
			\end{enumerate}
		Then it is the case that $dK_t \le dK'_t$.
	\end{proposition}

The intuition is quite clear. First, since one has $Y'_t \le Y_t$, if $Y$ doesn't touch the barrier ($Y_t > L_t$), then $dK_t=0$ and whether $Y'_t > L_t$ or  $Y'_t = L_t$, one has $dK'_t \ge 0 = dK_t $. So the only non-trivial case is when $Y$ touches the barrier, and therefore $Y'$ as well. In that case, since the extra forces $dK'$ and $dK$ are minimal, they only prevent the drifts $dV'$ and $dV$ from driving the solutions $Y'$ and $Y$ under the obstacle. But since $dV'_t \le dV_t$ in that case, the correction that could be needed for $Y$ will be less than that needed for $Y'$. The proof makes this heuristics rigorous.

Unlike in \cite{HamadeneAndLepeltierAndMatoussi}, \cite{PengAndXu}, \cite{LepeltierAndMatoussiAndXu}, \cite{KobylanskiAndLepeltierAndQuenezAndTorres}, the proof we give here works under minimal assumptions and in particular does not require a regularity assumptions on $f$, since it does not rely on the comparison theorem for BSDEs.

	\begin{proof}
		In this proof, contrary to the rest of the paper, $\Delta X$ denotes $X-X'$ for a generic quantity $X$. In order to deal with what happens locally when the process $\Delta Y$ touches $0$, we proceed as in El Karoui \textit{et al.} \cite{EKnKnPnPnQ} : write down the structure of $\Delta Y$ and $\Delta Y^+$, argue that these two processes are equal (since by assumption $\Delta Y \ge 0$), identify their finite variational and martingale parts, and then extract the relevant information. Our goal is to prove that $d\Delta K \le 0$. 
		
		We have
			\begin{align*}
				d\Delta Y &= -d\Delta V - d\Delta K + d\Delta N	\qquad\text{ and}		\\
				d(\Delta Y)^+ &= 1_{\{\Delta Y>0\}} d\Delta Y + \half dl \ ,
			\end{align*}
		where $l$ is the local time of $\Delta Y$ at 0. Identifying the finite variational and martingale parts, we see that 
			\begin{align*}
				-d\Delta V - d\Delta K  &= 1_{\{\Delta Y>0\}} \Big( -d\Delta V - d\Delta K \Big) + \half dl	\qquad\text{ and}		\\
				d\Delta N	 &= 1_{\{\Delta Y>0\}} d\Delta N	\ ,
			\end{align*}
		that is to say
			\begin{align*}
				1_{\{\Delta Y=0\}} \Big( -d\Delta V - d\Delta K \Big) &= \half dl	\qquad\text{ and}	\\
				1_{\{\Delta Y=0\}} d\Delta N	&= 0 \ .
			\end{align*}
		
		The second equation implies, by Itô's isometry and the orthogonality between $M$ and $\Delta N^\perp$, that  $1_{\Delta Y=0} \Big( \abs{\Delta Z \sigma}^2 dC + d\langle \Delta N^\perp \rangle \Big) =0 $. So we know that on the set $\{ Y' = Y \}$ (i.e. against $1_{\{\Delta Y=0\}}$) we have $Y=Y'$ and $Z=Z'$. 
		We also notice that by the Kunita-Watanabe inequality, $1_{\{\Delta Y=0\}} d\langle \nu',\Delta N^\perp \rangle = 0$ for any continuous semimartingale $\nu'$.
		
		The drift term can be rewritten, using $\Delta \nu = \nu - \nu = 0$,
			\begin{align*}
				d\Delta V_t &= \big( f(S) - f'(S') \big) dC + d\bracket{\nu, N^\perp} - d\bracket{\nu, (N')^\perp} + g d\bracket{N^\perp} - g' d\bracket{(N')^\perp}														\\
				&=  \big( f(S) - f(S') + (\Delta f)(S') \big) dC + d\bracket{\nu, \Delta N^\perp} + d\bracket{\Delta \nu	, (N')^\perp} 																										\\
							&\qquad\qquad + g \Big[ d\bracket{N^\perp} - d\bracket{(N')^\perp} \Big] + (\Delta g) d\bracket{(N')^\perp}								\\
				&= \Big[ \big( f(S) - f(S') \big)dC + d\bracket{\nu, \Delta N^\perp} + g d\bracket{N^\perp+(N')^\perp, \Delta N^\perp } \Big] 																					\\
							&\qquad\qquad + \Big[ (\Delta f)(S') dC + d\bracket{(\Delta \nu)	, (N')^\perp} + (\Delta g) d\bracket{(N')^\perp} \Big]																			\\
				&= \Big[ \big( f(S) - f(S') \big)dC + d\bracket{\nu', \Delta N^\perp}  \Big] + \Big[ d(\Delta V)(S') \Big] \ ,
			\end{align*}
		where $\nu' = \nu + \int g d(N^\perp + (N')^\perp)$. By the assumptions on the coefficients, we know that $d(\Delta V)(S')_t =: dI_t \ge 0$. 
		So we find that against $1_{\{\Delta Y_t=0\}}$ we have
			\begin{align*}
				1_{\{\Delta Y_t=0\}} d\Delta V_t = 0  + 1_{\{\Delta Y_t=0\}} dI_t \ .
			\end{align*}
		
		In the end, 
			\begin{align*}
				1_{\{\Delta Y=0\}} \big( -dI - d\Delta K \big) = \half dl \ ,
			\end{align*}
		so
			\begin{align*}
				1_{\{\Delta Y=0\}} d\Delta K = -1_{\{\Delta Y=0\}} \underbrace{dI}_{\ge 0} - \half \underbrace{dl}_{\ge 0} \le 0 \ ,
			\end{align*}
		and so we have proven that $1_{\{\Delta Y=0\}} d\Delta K \le 0$. And when $\Delta Y > 0$, one has $Y > Y' \ge L' = L$ so $dK=0 \le dK'$, and therefore $1_{\{\Delta Y>0\}} d\Delta K \le 0$, which completes the proof.
	\end{proof}

\section{Existence and stability.}					\label{section-existenceandstability}

In this section we work under the assumption that the derivatives of $f$ are controlled in the following way :
	\begin{itemize}
		\item[($\mathbf{A_{der}}$)] $f$ is twice continuously differentiable in the variables $(y,z)$ 
			and there exists $\rho,\rho',\lambda >0$, and $h\in L^2_{BMO}$ such that
			\begin{align*}
				\abs{f_{y}(t,y,z)} \le \rho r_t^2 \qquad \text{ and } \qquad \abs{f_{z}(t,y,z)} \le \rho' (h_t + \abs{z})		\ , \\
				\abs{f_{yy}(t,y,z)} \le \lambda r_t^2, \quad \abs{f_{yz}(t,y,z)} \le \lambda r_t \quad\text{ and }\quad \abs{f_{zz}(t,y,z)} \le \lambda \ .
			\end{align*}
	\end{itemize}
Rather than aiming to construct a solution to (\ref{equation-RBSDE.of.reference}) by an approximation procedure on the data, as was done in the Brownian setting (see \cite{KobylanskiAndLepeltierAndQuenezAndTorres}), 
we work in a more direct way, as in section 5 of \cite{EKnKnPnPnQ}, and for this we adapt the pertubation procedure introduced in \cite{Tevzadze} for BSDEs. We then analyze the dependence of the solution on the data.

\subsection{Principle.}

As said in the introduction, the strategy is to first apply the fixed point theorem. To perform this, one will use only the following assumption on $f$ :
	\begin{itemize}
		\item[$(\mathbf{A_{locLip}})$] The function $f$ is differentiable at $(0,0)$ (in $(y,z)$, for all $(\omega, s)$\,), and there exist $\lambda > 0$  such that, writing $\beta_s=f_y(s,0,0)$ and $\gamma_s=f_z(s,0,0)$, one has
			\begin{itemize}
				\item  for all $\omega, s, y_1,y_2, z_1, z_2$ : 
					\begin{multline*}
						\abs{f(s,y_1,z_1) - f(s,y_2,z_2) - \beta_s (y_1-y_2) - \gamma_s (z_1-z_2) } \\ 
							\le \lambda \Bigl( r_s \abs{y_1} + r_s \abs{y_2} + \abs{z_1} + \abs{z_2}  \Bigr) \bigl( r_s \abs{y_1-y_2} + \abs{z_1-z_2} \bigr) \ ,
					\end{multline*}
				\item $\gamma \in L^2_{BMO}$ and $\beta \in L^{\infty,1}$ (that is : $\int_0^T \abs{\beta_s} dC_s \in L^\infty$),
			\end{itemize}		
	\end{itemize}
which follows naturally from the assumption on the second derivative of $f$ in $(\mathbf{A_{der}})$. In all generality, it allows also for quadratic growth in $y$. So what one actually proves first is that when $f$ satisfies this assumption (with possibly quadratric growth in $y$ and $z$), and when the data are small enough (in a sense to specify), there exists a solution. 

\paragraph*{}
The perturbations procedure is then carried as follows for a reflected BSDE with obstacle $L \le 0$. Split the initial data in $n$ pieces : $(\xi^i)_{i=1 \ldots n}$ and $(\alpha^i)_{i = 1 \ldots n}$ such that $\sum_{i=1}^n \xi^i = \xi$ and $\sum_{i=1}^n \alpha^i = \alpha$, where $\alpha=f(\cdot,0,0)$, and such that for each $i$, $(\xi^i,\alpha^i)$ is small enough. For the sake of the proof we take the particular decomposition given by $\xi^i := \xi^{(n)} = \frac{1}{n} \xi$ and $\alpha^i := \alpha^{(n)} = \frac{1}{n} \alpha$, for $n$ big enough, though other decompositions would do.

First, there is a solution $S^1=(Y^1,N^1,K^1)$ to the reflected BSDE (\ref{equation-RBSDE.of.reference}) with small data $(f-\alpha+\alpha^1 , \nu , g, \xi^1,L)$.

Now, unless otherwise specified, we denote by $\overline{x}^k$ the sum $\Sigma_{j=1}^{k} x^j$, for a general quantity $x$ indexed by $\{ 1, \ldots, n \}$. 
For $i = 2$ to $n$, having obtained a solution $\overline{S}^{i-1}=(\overline{Y}^{i-1},\overline{N}^{i-1},\overline{K}^{i-1})$ to the reflected BSDE (\ref{equation-RBSDE.of.reference}) with parameters $(f-\alpha+\overline{\alpha}^{i-1} , \nu , g, \overline{\xi}^{i-1},L)$, one incorporates one more $(\alpha^i,\xi^i)$ in the system. One first constructs the perturbation $S^i=(Y^i,N^i,K^i)$ solving the pertubation equation  
	\begin{equation}  \label{equation-perturbation.equation.of.reference.procedure} 
		\left\{ \begin{aligned}
			dY^i &= -dV^i(Y^i,N^i) - dK^i + dN^i	\ ,\\
			Y^i_T &=\xi^i	\ ,\\
			\overline{Y}^{i-1} &+Y^i \ge \overline{L}^{i-1}+L^i \ ,\\
			d\overline{K}^{i-1}&+dK^i \ge 0 \ \text{and}\ (d\overline{K}^{i-1}+dK^i)(\overline{Y}^{i-1}+Y^i > \overline{L}^{i-1}+L^i)=0
		\end{aligned} \right.
	\end{equation}
with drift given by
	\begin{align*}
		dV^i(Y^i,N^i)_s &= \big[ f(\overline{S}^{i-1}+S^i) - f(\overline{S}^{i-1}) + \alpha^i_s \big]dC_s	+ d\langle \nu + \int 2g d(\overline{N}^{i-1})^\perp,(N^i)^\perp \rangle_s 
							+ gd\langle (N^i)^\perp \rangle_s			\\
		&=\big[ \overline{f}^{i-1}(S^i) + \alpha^i_s \big]dC_s	+ d\langle \overline{\nu}^{i-1},(N^i)^\perp \rangle_s + g d\langle (N^i)^\perp \rangle_s \ ,
	\end{align*}	
where $\overline{\nu}^{i-1} = \nu + \int 2g d(\overline{N}^{i-1})^\perp$, and $\overline{f}^{i-1}$ is the function $f$ recentered around $\overline{S}^{i-1}$. It satisfies $\overline{f}^{i-1} (0)=0$ so the residual-drift (constant part) in this equation is given by $\alpha^i$. So the parameters $(\overline{f}^{i-1} + \alpha^i , \nu^{i-1}, g, \xi^i , L)$ here are small in the required sense. Finally, one sums $\overline{S}^i := \overline{S}^{i-1} + S^i$ to obtain a solution to the reflected BSDE (\ref{equation-RBSDE.of.reference}) of parameters $(f-\alpha+\overline{\alpha}^{i} , \nu , g, \overline{\xi}^{i},L)$. For $i=n$ this provides a solution to the reflected BSDE of interest. 

This allows us to conclude to existence for those reflected BSDEs with negative obtacles. But then we can show that up to translation, this covers all the cases where the obstacle is upper-bounded.

Note already that the above perturbation equation (\ref{equation-perturbation.equation.of.reference.procedure}) is not a RBSDE in the variable $S^i=(Y^i,N^i,K^i)$ because $K^i$ is not necessarily increasing. It could be viewed as a reflected BSDE in the variable $(Y^i,N^i,\overline{K}^i)$ but this point of view will not be used (see the remark after proposition \ref{proposition-smallperturbationtoaRBSDE} and remark \ref{remark-uniqueness.for.perturbations.} after its proof). Also, note that the solution $S^1$ to the initial, small RBSDE can be viewed as a perturbation : $\overline{S}^1=0+S^1$, $0$ being the solution to the RBSDE of parameters $(f-f(\cdot,0,0), \nu , g,0,L)$. So it would be enough to study only the pertubation equations, but it seemed clearer to treat first the small reflected BSDEs and then deal with what changes for the perturbation equations.

\subsection{Existence for small reflected BSDEs.}

\subsubsection{Underlying problem.}

In order to use the fixed point theorem, we need to check that the underlying problem, that is to say the backward stochastic problem that one sees when the drift $dV_t$ is a fixed process and doesn't depend on the solution, defines indeed a map from $\S^\infty \times BMO$ to itself. For reflected BSDEs, as was explained in El Karoui \emph{et al.} \cite{EKnKnPnPnQ}, the solution is the Snell envelope of a certain process (more precisely, $Y+\int_0^\cdot dV_s$ is the Snell envelope of $L + \int_0^\cdot dV_s$). 

	\begin{proposition}			\label{proposition-underlyingproblem}
		Let $V \in L^1_{BMO}$ (in the sense that $\sup_t \norm{ E\big( \int_t^T \abs{dV_s} \big|\F_t \big) }_\infty < +\infty$), $\xi \in L^\infty$, and $L$ be upper bounded. There exist a unique $(Y,N,K) \in \mathcal{S}^\infty \times BMO \times \mathcal{A}$ solution to the reflected BSDE : 
			\begin{align}	\label{equation-RBSDE.with.non-circular.driver} 
				\left\{ \begin{aligned}
					dY &= -dV - dK + dN	\ ,\\
					Y_T &=\xi	\ ,\\
					Y &\ge L \ \text{and}\ 1_{\{Y>L\}}dK=0 \ .
				\end{aligned} \right.
			\end{align}			
		In particular, this applies when $dV_s = dV(y,n)_s = f(s,y_s,z_s \sigma_s)dC_s + d\langle \nu,n^\perp \rangle_s + g_s d\langle n^\perp \rangle_s$, for $f$ satisfying the quadratic growth condition $(\mathbf{A_{qg}})$, $\nu \in BMO$, $g \in L^\infty$ and $(y,n) \in \mathcal{S}^\infty \times BMO$.
	\end{proposition}

	\begin{proof}
		We know from proposition 5.1 in El Karoui \textit{et al.} \cite{EKnKnPnPnQ} that $Y_t$ is given by 
			\begin{align} \label{equation-explicit.solution.underlying.problem}
				Y_t = \esssup_{\tau \in \mathcal{T}_t^T} E\bigg( \int_t^\tau dV_s + L_\tau 1_{\tau<T} + \xi 1_{\tau=T} \Big| \F_t \bigg) \ ,
			\end{align}
		where $\mathcal{T}_t^T$ are the stopping times $\tau$ such that $t\le\tau\le T$, and that the square integrable martingale $N$ and the increasing process $K$ are the Doob-Meyer decomposition of the supermartingale $Y+V$. Our goal is to check that 
		$(Y,N)$ is indeed in $\S^\infty \times BMO$.
		For an upper bound on $Y_t$, we have
			\begin{align*}
				E\bigg( \int_t^\tau dV_s +  L_\tau 1_{\tau<T} + \xi 1_{\tau=T} \Big| \F_t \bigg)  &\le E\bigg( \int_t^T \abs{dV_s} \Big| \F_t \bigg) + E\bigg( L_\tau^+ \Big| \F_t \bigg) 	
							+ E\bigg( \xi^+ \Big| \F_t \bigg) \\
				&\le \norm{V}_{L^1_{BMO}} + \norm{L^+}_{\infty} + \norm{\xi^+}
			\end{align*}
		for any stopping time $\tau$, so $Y_t \le \norm{V}_{L^1_{BMO}} + \norm{L^+}_{\infty} + \norm{\xi^+}$. 
		For a lower bound, since $Y$ solves (\ref{equation-RBSDE.with.non-circular.driver}), and using the fact that $K$ is increasing, we have
			\begin{align*}
				Y_t &= E\bigg( \xi + \int_t^T dV_s + \big(K_T - K_t \big) \Big| \F_t \bigg)														\\
				&\ge	 E\bigg( \xi + \int_t^T dV_s  \Big| \F_t \bigg)															\\
				&\ge -\norm{\xi^-}_\infty - \norm{V}_{L^1_{BMO}} \ ,
			\end{align*}	
		so $Y$ is indeed in $\S^\infty$.		
		One can then invoke remark \ref{remark-aprioriestimate} after proposition \ref{proposition-aprioriestimate} to conclude that $N \in BMO$. 
		
		\paragraph*{} 
		We now prove the second assertion. For a drift process $V$ of the form described above,
			\begin{align*}
				\abs{dV_s} \le \abs{f(s,y_s,z_s\sigma_s)} dC_s + \abs{d\bracket{\nu,n^\perp}_s} + \abs{g_s} \, \abs{d\bracket{n^\perp}_s}.
			\end{align*}
		Using the assumption $(\mathbf{A_{qg}})$ on $f$ and the Kunita-Watanabe inequality, 
		we have, similarly as in proposition \ref{proposition-aprioriestimate}, 
			\begin{align*}
				\Et{ \int_t^T \abs{dV_s} } &\le \lambda(\norm{y}_{\S^\infty}) \Et{ \int_t^T h_s^2 + \abs{z_s\sigma_s}^2 dC_s}  \\
									&\qquad + \Et{ \int_t^T \abs{d\bracket{\nu, n^\perp}_s} } + \norm{g}_\infty \Et{ \int_t^T d\bracket{n^\perp}_s }	\\
					&\le \Lambda \norm{h}_{L^2_{BMO}}^2 + \Big( \Lambda + 1 + \norm{g}_\infty \Big) \norm{n}_{BMO}^2 + \norm{\nu}_{BMO}^2 ,
			\end{align*}
		where $\Lambda = \lambda(\norm{y}_{\S^\infty}) $. Hence $V \in L^1_{BMO}$ as wanted.
	\end{proof}
	
\subsubsection{Existence for RBSDEs with small data.}

First one proves that there is a solution when the data are small and when, essentially, the drift is purely quadratic in the solution.

	\begin{proposition}				\label{proposition-smallRBSDE.purelyquadratic}
		Let $\lambda > 0$. Let $f$ satisfy assumption $(\mathbf{A_{locLipz}})$, with parameters $(\beta=0, \gamma=0, \lambda, r)$ and be such $\alpha = f(\cdot,0,0) \in L^{\infty,1}$ (i.e. : $\int_0^T \abs{\alpha_s}dC_S \in L^\infty $). Let $\nu=0 \in BMO$ and $g$ be bounded by $\lambda$. There exists $\epsilon_0=\epsilon_0(\lambda,r) > 0$ such that if the size of the data
			\begin{align*}
				\boldsymbol{D} =  \norm{\xi}_{\infty} + \norm{f(\cdot,0,0)}_{\infty,1} + \norm{L^+}_\infty \le \epsilon_0 \ ,
			\end{align*}
		then there exists a solution $S=(Y,N,K) \in \mathcal{S}^\infty \times BMO(P) \times \mathcal{A}$ to the reflected BSDE (\ref{equation-RBSDE.of.reference}) with data $(V,\xi,L)$, where $dV(Y,N)_s =  f(s,Y_s,Z_s \sigma_s)dC_s + g_s d\langle N^\perp \rangle_s$ .
		
		More precisely,  
			\begin{align*}
				\epsilon_0(\lambda,r) = \frac{1}{2^{10} \lambda \big( \norm{r}_{\infty,2}^2 + 2 \big)} \ .
			\end{align*}
		
		Also, for any $R \le R_0(\lambda,r) = \frac{1}{2^{5} \lambda \big( \norm{r}_{\infty,2}^2 + 2\big)}$, if $\boldsymbol{D} \le \frac{R}{2^{5}}$, then this solution is known to satisfy
			\begin{align*}
				\norm{S}^2 = \norm{Y}^2_{\S^\infty} + \norm{N}^2_{BMO(P)} \le R^2 \ .
			\end{align*}
	\end{proposition}
	
	\begin{proof}
		We study the map $Sol : \mathcal{S}^\infty \times BMO \rightarrow \mathcal{S}^\infty \times BMO$ which sends $(y,n)$ on the solution $(Y,N)$ to the reflected BSDE  
			\begin{align} \label{equation-RBSDE.with.token.variables.in.driver} 
				\left\{ \begin{aligned}
					dY &= -dV(y,n) - dK + dN	\ ,\\
					Y_T &=\xi	\ ,\\
					Y &\ge 0 \ \text{and}\ 1_{\{Y>0\}} dK=0 \ ,
				\end{aligned} \right.
			\end{align}
			where $dV(y,n)_s =  f(s,y_s,z_s \sigma_s)dC_s + g_s d\langle n^\perp \rangle_s$ .
		This map is well defined according to proposition \ref{proposition-underlyingproblem}, and $(Y,N) \in \mathcal{S}^\infty \times BMO$ is a solution of (\ref{equation-RBSDE.of.reference}) if and only if it is a fixed point of $Sol$. It will be seen that $Sol$ is not a contraction on the whole space, but it is on a small ball, and it stabilizes such a small ball if the data are small enough. Therefore there exists at least one fixed point in the space.
		
		\paragraph*{}		
		We study first the regularity of $Sol$. Take $s=(y,n)$ and $s'=(y',n')$ in $\mathcal{S}^\infty \times BMO$, write $S=Sol(s)$, $S'=Sol(s')$, and $\Delta x=x'-x$ for a generic quantity $x$. The  semimartingale decomposition of $\Delta Y$ is $d\Delta Y = -d\Delta V - d\Delta K + d\Delta N$, and the terminal value is $0$. Therefore, applying Itô's formula to $(\Delta Y)^2$ between $t \in \mathcal{T}_0^T$ and $T$, and taking the expectation conditional to $\F_t$ one has, since $\int_0^\cdot \Delta Y d\Delta N$ is a martingale,
			\begin{align} \label{equation-small.Ito.fixedpoint.DeltaSol}
				(\Delta Y_t)^2 + E \Big( \int_t^T d\langle \Delta N \rangle_s | \F_t \Big) &= 0^2 + 2  E \Big( \int_t^T \Delta Y_s d\Delta V_s | \F_t \Big) \\
					&\quad + 2  E \Big( \int_t^T \Delta Y_s d\Delta K_s  | \F_t \Big) - 0 \ .																									\notag
			\end{align} 
		Let us now look at the third term on the right-hand side. Using the fact that $YdK=LdK$ and $Y'dK'=LdK'$ one has 
			\begin{align*}
				\Delta Y d\Delta K &= (Y'-Y)dK' - (Y'-Y)dK																								\\
				&= \underbrace{(L-Y)}_{\le 0}\underbrace{dK'}_{\ge 0} - \underbrace{(Y'-L)}_{\ge 0}\underbrace{dK}_{\ge 0}	\le 0 \ .
			\end{align*}
		Let us now deal with the second term : 
			\begin{align*}
				E \Big( \int_t^T \Delta Y_s d\Delta V_s | \F_t \Big) \le \norm{\Delta Y}_\infty E \Big( \int_t^T \abs{d\Delta V_s} | \F_t \Big) \ .
			\end{align*}
		The assumption on $f$ gives
			\begin{align*}
				\abs{d\Delta V_s} &\le \lambda \Big( r_s\abs{y_s}+r_s\abs{y'_s}+\abs{z_s \sigma_s} +\abs{z'_s \sigma_s} \Big) \big( r_s\abs{\Delta y_s} + \abs{\Delta z_s \sigma_s} \big) dC_s	\\
					&\qquad  + \abs{g_s} \abs{d \langle \Delta n^\perp, n^\perp+(n')^\perp \rangle_s} .
			\end{align*}
		Consequently, using the Cauchy-Schwartz and the Kunita-Watanabe inequalities, and the elementary inequality $(\sum_{i=1}^n a_i)^2 \le n \sum a_i^2$, we have
			\begin{align*}
				E \bigg(& \int_t^T \abs{d\Delta V} \Big| \F_t \bigg)			\\
					& \le 2^{\frac{3}{2}}\lambda \Et{ \int_t^T r_s^2\abs{y_s}^2 +r_s^2\abs{y'_s}^2 + \abs{z_s\sigma_s}^2 + \abs{z'_s\sigma_s}^2 \ dC_s }^\half 
							\Et{ \int_t^T r_s^2 \abs{\Delta y_s}^2 	+ \abs{\Delta z_s \sigma_s}^2 \ dC_s }^\half \\
											&\qquad + \norm{g}_\infty \Et{ \int_t^T d\bracket{\Delta n^\perp}_s}^{\half} \Et{ \int_t^T d\bracket{n^\perp + (n')^\perp}_s }^{\half}\ .
			\end{align*}
		Now, by orthogonality, one has $\abs{z_s\sigma_s}^2 dC_s + d\bracket{n^\perp}_s = d\bracket{n}_s$ so in particular each term on the left-hand side of this equation is smaller than or equal to the right-hand side. So
			\begin{align*}
				E &\bigg( \int_t^T \abs{d\Delta V} \Big| \F_t \bigg) 	\\
					& \le 2^{\frac{3}{2}}\lambda \Et{ \int_t^T \Big( r_s^2\abs{y_s}^2 +r_s^2\abs{y'_s}^2\Big) dC_s + d\bracket{n}_s + d\bracket{n'}_s  }^\half 
							\Et{ \int_t^T r_s^2 \abs{\Delta y_s}^2 dC_s + d\bracket{\Delta n}_s  }^\half \\
											&\qquad + \norm{g}_\infty \Et{ \int_t^T d\bracket{\Delta n}_s}^{\half} \Et{ \int_t^T d\bracket{n+(n')}_s }^{\half} \\
					&\le 2^{\frac{3}{2}}\lambda \Big( \norm{r}_{\infty,2}^2 \norm{y}_{\S^\infty}^2 + \norm{r}_{\infty,2}^2 \norm{y'}_{\S^\infty}^2+ \norm{n}_{BMO}^2 + \norm{n'}_{BMO}^2  \Big)^\half	
											\Big( \norm{r}_{\infty,2}^2 \norm{\Delta y}_{\S^\infty}^2 + \norm{\Delta n}_{BMO}^2 \Big)^\half		\\
								&\qquad + \norm{g}_\infty \norm{\Delta n}_{BMO} \ \norm{n+n'}_{BMO}	\\
					&\le 2^{\frac{3}{2}}\lambda\  \big( \norm{r}_{\infty,2}^2 + 1 \big) \Big( \norm{y}_{\S^\infty}^2 + \norm{n}_{BMO}^2 + \norm{y'}_{\S^\infty}^2 + \norm{n'}_{BMO}^2 \Big)^\half 
											\Big( \norm{\Delta y}_{\S^\infty}^2 + \norm{\Delta n}_{BMO}^2 \Big)^\half \\
								&\qquad + \norm{g}_\infty \Big( \norm{n}_{BMO} + \norm{n'}_{BMO} \Big) \norm{\Delta n}_{BMO}  .
			\end{align*}
		Now, by definition of the norm on $\S^\infty \times BMO$, $\norm{y}_{\S^\infty}^2+\norm{n}_{BMO}^2=\norm{s}^2$. 
		Again, this implies in particular that $\norm{n}_{BMO}^2 \le \norm{s}^2$. So, recalling that $\norm{g}_\infty \le \lambda$, using $(a^2+b^2)^\half \le a + b$ and majorizing $1 \le 2^{\frac{3}{2}}$ (for the 2nd inequality), we have
			\begin{align*}
				E \bigg( \int_t^T \abs{d\Delta V} \Big| \F_t \bigg) 	
					&\le 2^{\frac{3}{2}}\lambda\  \big( \norm{r}_{\infty,2}^2 + 1 \big) \Big( \norm{s}^2 + \norm{s'}^2  \Big)^\half \Big( \norm{\Delta s}^2 \Big)^\half 
											+ \lambda \Big( \norm{s} + \norm{s'} \Big) \norm{\Delta s} \\
					&\le 	2^{\frac{3}{2}}\lambda\  \big( \norm{r}_{\infty,2}^2 + 1 \big) \big( \norm{s} + \norm{s'} \big) \norm{\Delta s}		
											+ 2^{\frac{3}{2}}\lambda\  \big( \norm{s} + \norm{s'} \big) \norm{\Delta s}  \\
					&\le 2^{\frac{3}{2}}\lambda \big( \norm{r}_{\infty,2}^2 + 1 + 1 \big)\big( \norm{s'} + \norm{s} \big) \norm{\Delta s} \ .
			\end{align*}
		Equation (\ref{equation-small.Ito.fixedpoint.DeltaSol}) then yields, using $2ab \le \frac{1}{4}a^2 + 4b^2$ and $(a+b)^2 \le 2 (a^2+b^2)$,
			\begin{align*}
				(\Delta Y_t)^2 + \Et{ \int_t^T d \langle \Delta N \rangle_s} &\le \frac{1}{4}\norm{\Delta Y}_{\S^\infty}^2 \\
							&\qquad + 4 \times 2^3\lambda^2 \big( \norm{r}_{\infty,2}^2 + 2 \big)^2 \times 2 \big(\norm{s'}^2 + \norm{s}^2\big)\norm{\Delta s}^2 \ ,
			\end{align*}
		and by taking the sup, we finally find, since $\norm{\Delta Y}_{\S^\infty} \le \norm{\Delta S}$, that
			\begin{align} \label{equation-small.fixedpoint.DeltaSol.final.estimate} 
				\norm{\Delta S}^2 \le 2^8 \lambda^2 \big( \norm{r}_{\infty,2}^2 + 2 \big)^2 \big( \norm{s}^2 +\norm{s'}^2 \big) \norm{\Delta s}^2 \ .
			\end{align}			

		\paragraph*{}
		Let us now study the size of $S=Sol(s)$. Following the very same computations and arguments as for $\Delta S$ we have first 
			\begin{align} \label{equation-small.Ito.fixedpoint.SizeSol}
				(Y_t)^2 + E \Big( \int_t^T d\langle N \rangle_s | \F_t \Big)  \le \norm{\xi}_\infty^2 +  2 E \bigg( \int_t^T Y_s d V_s \Big| \F_t \bigg) + 2 E \bigg( \int_t^T Y_s d K_s \Big| \F_t \bigg).
			\end{align}	
		Since $YdK=LdK$ and, importantly, since $K$ is increasing, one can write
			\begin{align*}
				\int_t^T Y_s dK_s = \int_t^T L_s dK_s &\le \norm{L^+}_{\S^\infty} \big( K_T - K_t \big)	 = \norm{L^+}_{\S^\infty} \bigg( Y_t - \xi  -\int_t^T dV + (N_T-N_t) \bigg) \ ,	
			\end{align*}
		so that
			\begin{align*}
				\Et{\int_t^T YdK} \le \norm{L^+}_{\S^\infty} \abs{Y_t} + \norm{L^+}_{\S^\infty} \norm{\xi}_\infty + \norm{L^+}_{\S^\infty}  \Et{\int_t^T \abs{dV}} + 0	\ .
			\end{align*}
		Reinjecting this into (\ref{equation-small.Ito.fixedpoint.SizeSol}), then using the Young inequality, in particular the case $2ab \le 8a^2+\frac{1}{8}b^2$, leads to 
			\begin{align*} 
				(Y_t)^2 + E \bigg( \int_t^T d\langle N \rangle_s | \F_t \bigg)  \le 
										\Big( 2\norm{\xi}_\infty^2 + 10 \norm{L^+}_{\S^\infty}^2 \Big) + \frac{1}{4}\norm{Y}_{\S^\infty}^2 + 9 \Et{\int_t^T \abs{dV}}^2 \ .
			\end{align*}
		Now, by the assumption on $f$, 
			\begin{align*}
				\abs{dV_s} &\le \Big[ f(s,0,0) + \lambda \big( r_s\abs{y_s} + \abs{z_s \sigma_s} \big)^2  \Big]dC_s + \abs{g_s}d\langle n^\perp \rangle_s \\
					&\le \Big[ f(s,0,0) + 2 \lambda \big( r_s^2 \abs{y_s}^2 + \abs{z_s \sigma_s}^2 \big)  \Big]dC_s + \abs{g_s}d\langle n^\perp \rangle_s 
			\end{align*}	
		so, by the same argumentation as for $\Delta V$ above, 
			\begin{align*}
				\Et{\int_t^T \abs{dV}} &\le \norm{f(\cdot,0,0)}_{\infty,1} + 2 \lambda \big( \norm{r}_{\infty,2}^2 + 1 \big) \norm{s}^2 + \lambda \norm{s}^2			\\
				&\le \norm{f(\cdot,0,0)}_{\infty,1} + 2 \lambda \big( \norm{r}_{\infty,2}^2 + 2 \big) \norm{s}^2 .
			\end{align*}
		Consequently, after taking $\sup_t$ and using $\norm{Y}_{\S^\infty} \le \norm{S}$, one has
			\begin{align*}
				\norm{S}^2  \le \Big( 4 \norm{\xi}_\infty^2 + 20 \norm{L^+}_{\S^\infty}^2 \Big) + \frac{1}{2}\norm{S}^2 
										+ 18 \times 2 \times \Big[ \norm{f(\cdot,0,0)}_{\infty,1}^2 + 2^2 \lambda^2 \big( \norm{r}_{\infty,2}^2 + 2 \big)^2 \norm{s}^4 \Big] \ .
			\end{align*}		
		Collecting the terms in $\norm{S}^2$ and majorizing largely one has finally
			\begin{align} \label{equation-small.fixedpoint.SizeSol.final.estimate}
				\norm{S}^2 &\le 2^9 \boldsymbol{D}^2 + 2^9 \lambda^2 \big( \norm{r}_{\infty,2}^2 + 2 \big)^2\norm{s}^4		\ ,
			\end{align}
		where $\boldsymbol{D} = \norm{\xi}_\infty + \norm{L^+}_{\S^\infty} + \norm{f(\cdot,0,0)}_{\infty,1}$ and we used $a^2+b^2+c^2 \le (a+b+c)^2$.
		
		To have $Sol$ be a contraction on a closed (and therefore complete) ball $\overline{B}(0,R)$ of $\mathcal{S}^\infty \times BMO$, 
		we see from (\ref{equation-small.fixedpoint.DeltaSol.final.estimate}) and (\ref{equation-small.fixedpoint.SizeSol.final.estimate}) 
		that we would like the radius $R$ and the size $\boldsymbol{D}$ of the data
		to be sufficiently small so that  
		$2^{9} \lambda^2 \big( \norm{r}_{\infty,2}^2 + 2 \big)^2 R^2 \le \half (<1)$ and $2^9 \boldsymbol{D}^2 + 2^9 \lambda^2 \big( \norm{r}^2 + 2 \big)^2 R^4 \le R^2$. This is the case as soon as
			\begin{align*}
				R &\le R_0(\lambda,r) := \frac{1}{2^5 \lambda \big( \norm{r}_{\infty,2}^2 + 2 \big)} \\
				\boldsymbol{D} &\le \frac{R}{2^5} \le \frac{R_0(\lambda,r)}{2^5} =: \epsilon_0(\lambda,r) \ .
			\end{align*}
	\end{proof}

We now remove the assumption that the linear terms in the drift are null. 

	\begin{proposition}				\label{proposition-smallRBSDE}
		Let $\lambda > 0$. Let $f$ satisfy assumption $(\mathbf{A_{locLipz}})$, with parameters $(\beta, \gamma, \lambda, r)$ and be such that $\alpha = f(\cdot,0,0) \in L^{\infty,1}$ (i.e. : $\int_0^T \abs{f(s,0,0)}dC_s \in L^\infty $). Let $\nu \in BMO$ and $g$ be bounded by $\lambda$. There exists $\epsilon_0=\epsilon_0(\beta,\lambda,r) > 0$ such that if the size of the data
			\begin{align*}
				\boldsymbol{D} =  \norm{\xi}_{\infty} + \norm{f(\cdot,0,0)}_{\infty,1} + \norm{L^+}_\infty \le \epsilon_0 \ ,
			\end{align*}
		then there exists a solution $S=(Y,N,K) \in \mathcal{S}^\infty \times BMO(P) \times \mathcal{A}$ to the reflected BSDE (\ref{equation-RBSDE.of.reference}) with data $(f,\nu,g,\xi,L)$.
		
		More precisely, 
			\begin{align*}
				\epsilon_0(\beta,\lambda,r)  = \frac{e^{-2\norm{\beta}_{\infty,1}}}{2^{10} \lambda \big( \norm{r}_{\infty,2}^2 + 2 \big)} \ .
			\end{align*}
		
		Also, for any $R \le R_0(\widehat{\lambda},r) = \frac{1}{2^{5} \widehat{\lambda} \big(  \norm{r}_{\infty,2}^2 + 2 \big)}$, 
		where $\widehat{\lambda}=\exp\big(\norm{\beta}_{\infty,1}\big) \lambda$, 
		if $\boldsymbol{D} \le \exp(-\norm{\beta}_{\infty,1}) \frac{R}{2^{5}}$, then this solution is known to satisfy
			\begin{align*}
				\norm{\widehat{S}}_Q^2 = \norm{\widehat{Y}}^2_{\S^\infty} + \norm{\widehat{\widetilde{N}}}^2_{BMO(Q)} \le R^2 \ ,
			\end{align*}
		where $\widehat{Y}_t=e^{\int_0^t \beta_u dC_u} Y_t$ and $\widetilde{\widehat{N}}$ is the martingale part of $\widehat{Y}$ under $Q$ : $\frac{dQ}{dP}=\mathcal{E}(\int \gamma \sigma^{-1} dM + \nu)$.
	\end{proposition}
	
	\begin{proof}
		Write $f(t,y,z)= \beta_t y + \gamma_t z + h(t,y,z)$, where $\beta_t=f_y(t,0,0)$ and $\gamma_t=f_z(t,0,0)$  (so that $h(t,0,0)=f(t,0,0)=\alpha_t$). Note that $h$ satisfies $(\mathbf{A_{locLipz}})$ with parameters $(\beta=0, \gamma=0, \lambda, r)$.

		The idea is that if $(Y,N,K)$ is a solution to the reflected BSDE (\ref{equation-RBSDE.of.reference}), 
		one can eliminate the linear terms $(\beta_t Y_t + \gamma_t Z_t\sigma_t)dC_t+d\bracket{\nu,N^\perp}_t $ in the drift $dV(Y,N)_t$ by a pair of transforms and obtain a reflected BSDE with purely quadratic drift. Proposition \ref{proposition-smallRBSDE.purelyquadratic} guarantees the existence of a solution to such a RBSDE, so undoing the transforms yields a solution to (\ref{equation-RBSDE.of.reference}).

		In view of this, let us define the measure $Q$ by $\frac{dQ}{dP} = \Exp(L)$ where $L = \int \gamma \sigma^{-1} dM + \nu$. Then $\widetilde{M} := M - \langle L,M \rangle = 	M - \int \gamma \sigma^* dC$ is a $BMO(Q)$-martingale. Define also $B=\exp\Big( \int_0^\cdot \beta_u dC_u \Big)$, which is a bounded process. Define the transformed data
			\begin{align*}
				&\widehat{h}(s,y,z) = B_s h(s,B_s^{-1} y, B_s^{-1} z)							\ ,\\
				&\widehat{g}_s = B_s^{-1} g_s															\ ,\\
				&\widehat{\xi} = B_T\xi																	\ ,\\
				&\widehat{L} = B L \ .				
			\end{align*}
		Note that $\widehat{h}$ satisfies $(\mathbf{A_{locLipz}})$ with parameters $(\beta=0,\gamma=0,\widehat{\lambda},r)$ where $\widehat{\lambda} = \lambda \exp(\norm{\beta}_{\infty,1})$. Proposition \ref{proposition-smallRBSDE.purelyquadratic} ensures the existence of a solution $(\widehat{Y}, \widehat{\widetilde{N}}, \widehat{K} ) \in \mathcal{S}^\infty \times BMO(Q) \times \mathcal{A} $ under $Q$ to the reflected BSDE (\ref{equation-RBSDE.of.reference}) with transformed data $(\widehat{h}, \nu=0, \widehat{g}, \widehat{\xi}, \widehat{L})$.
		Indeed,
			\begin{align*}
				&\norm{\widehat{g}}_\infty \le \exp(\norm{\beta}_{\infty,1}) \norm{g}_\infty \le  \exp(\norm{\beta}_{\infty,1}) \lambda = \widehat{\lambda}	<+\infty		\ ,\\
				&\norm{\widehat{h}(\cdot,0,0)}_{\infty,1} \le \exp(\norm{\beta}_{\infty,1}) \norm{f(\cdot,0,0)}_{\infty,1} < +\infty \ ,\\
				&\norm{\widehat{\xi}}_\infty \le \exp(\norm{\beta}_{\infty,1}) \norm{\xi}_\infty	<+\infty	\ ,\\
				&\norm{\widehat{L^+}}_{\S^\infty} \le \exp(\norm{\beta}_{\infty,1}) \norm{L^+}_{\S^\infty} <+\infty \ ,
			\end{align*}
		so if $\boldsymbol{D} \le \exp(- \norm{\beta}_{\infty,1}) \epsilon_0(\widehat{\lambda},r)=\exp(- 2\norm{\beta}_{\infty,1}) \epsilon_0(\lambda,r)$, proposition \ref{proposition-smallRBSDE.purelyquadratic} applies.
		
		Now, define $Y = B^{-1} \widehat{Y}$, $\widetilde{N}=\int_0^\cdot B^{-1}d\widehat{\widetilde{N}}=\int \widetilde{Z}d\widetilde{M} + \widetilde{N}^\perp$ and $K=\int_0^\cdot B^{-1}d\widehat{K}$. The Girsanov $(Q \rightarrow P)$-transform of $\widetilde{N}$,
			\begin{align*}
				N &= \widetilde{N} + \bracket{ L,\widetilde{N}} \\
				&= \int \widetilde{Z}d\widetilde{M} + \widetilde{N}^\perp + \int \gamma\widetilde{Z}\sigma dC + \bracket{\nu,\widetilde{N}^\perp}	\\
				&= \int \widetilde{Z}dM + N^\perp \ ,
			\end{align*}
		is a $BMO(P)$-martingale. $Y$ is a bounded semimartingale, since $B^{-1}$ is bounded, and differentiating $Y = B^{-1} \widehat{Y}$ shows that $(Y,N,K)$ is a solution to the reflected BSDE (\ref{equation-RBSDE.of.reference}) with data $(f,\nu,g,\xi,L)$, as we wanted.
	\end{proof}

\subsection{Perturbation of a reflected BSDE.}

We now deal with existence for perturbation equations like (\ref{equation-perturbation.equation.of.reference.procedure}). We assume we have a solution $S^1=(Y^1,N^1,K^1)$ to a reflected BSDE with data $(f,\nu,g,\xi^1,L^1)$, and want to construct a solution $\overline{S}^2$ to a reflected BSDE with slightly different data $(f+\alpha^2,\nu,g,\xi^1+\xi^2,L^1+L^2)$. The idea is to construct the difference $S^2 = (Y^2,N^2,K^2)=\overline{S}^2-S^1$. The next proposition shows how this can be done despite the fact that $K^2$ is not an increasing process,  so long as one does not change the obstacle ($L^2=0$).

	\begin{proposition}			\label{proposition-smallperturbationtoaRBSDE}
		Let $f$ satisfy $(\mathbf{A_{der}})$ with parameters $(\rho,\rho',\lambda,r,h)$ and be such that $\alpha=f(\cdot,0,0) \in L^{\infty,1}$, let $g \in L^\infty$ be bounded by $\lambda$ and $\nu \in BMO$. Let also $\xi^1 \in L^\infty$ and $L^1$ be upper bounded. Assume that there exists a solution $S^1=(Y^1,N^1,K^1)$ to the RBSDE (\ref{equation-RBSDE.of.reference}) with data $( f,g,\nu, \xi^1,L^1)$. Now let $\xi^2 \in L^\infty$ and $\alpha^2 \in L^{\infty,1}$ (and $L^2=0$). If
			\begin{align*}
				\delta \boldsymbol{D} = \norm{\xi^2}_{\infty} + \norm{\alpha^2}_{\infty,1} \le \epsilon_0(\rho,2\lambda,r)  
										= \frac{e^{-2\rho\norm{r}_{\infty,2}^2}}{2^{10} (2\lambda) \big( \norm{r}_{\infty,2}^2 + 2 \big)},
			\end{align*}
		then there exist $S^2=(Y^2,N^2,K^2)$ where $Y^2 \in \mathcal{S}^\infty$, $N^2 \in BMO(P)$ and $K^2$ has finite variation, solving the perturbation equation 
			\begin{equation} \label{equation-perturbation.equation.of.reference.notationY2}
				\left\{ \begin{aligned}
					dY^2 &= -dV^2(Y^2,N^2) - dK^2 + dN^2	\ ,\\
					Y^2_T &=\xi^2	\ ,\\
					Y^1 &+Y^2 \ge L^1 + L^2 \ ,\\ 
					dK^1 &+ dK^2 \ge 0 \ \text{and}\ 1_{\{Y^1+Y^2>L^1+L^2\}}(dK^1+dK^2)=0 
				\end{aligned} \right.
			\end{equation}
		with drift given by
			\begin{align*}
				dV^2(Y^2,N^2)_s = \big[ f(s,Y^2_s + Y^1_s,Z^2_s \sigma_s + Z^1_s \sigma_s) - f(s,Y^1_s,Z^1_s \sigma_s) + \alpha^2_s \big]dC_s \\
					\qquad + d\Big\langle \nu + \int 2g d(N^1)^\perp,(N^2)^\perp \Big\rangle_s + g_s d\Big\langle (N^2)^\perp \Big\rangle_s.
			\end{align*}
		
		So $\overline{S}^2 :=S^1 + S^2$ is a solution to the RBSDE (\ref{equation-RBSDE.of.reference}) with data $( f+\alpha^2,g,\nu, \xi^1+\xi^2,L^1)$.
		
		We further know that for any $R \le R_0(\widehat{2\lambda},r) = \frac{1}{2^{5} \widehat{2\lambda} \big(  \norm{r}_{\infty,2}^2 + 2 \big)}$, 
		where $\widehat{2\lambda}=2\lambda \exp\big(\rho \norm{r}_{\infty,2}^2\big) $, 		
		if $\delta \boldsymbol{D} \le \exp(-\rho\norm{r}_{\infty,2}^2) \frac{R}{2^{5}}$, then this solution satisfies
			\begin{align*}
				\norm{\widehat{S^2}}_Q^2 = \norm{\widehat{Y^2}}^2_{\S^\infty} + \norm{\widehat{\widetilde{N^2}}}^2_{BMO(Q)} \le R^2 \ .
			\end{align*}
	\end{proposition}

Note that while $S^2=(Y^2,N^2,K^2)$ is not the solution to a reflected BSDE, $(Y^2,N^2,\overline{K}^2)$ is. However the drift there would be $dV^2(Y^2,N^2)_s-dK^1_s$, whose residual action (when $(Y^2,N^2)=(0,0)$) is $\alpha^2_s dC_s -dK^1_s$, and this has no reason to be small. We can therefore not simply invoke proposition \ref{proposition-smallRBSDE} to construct $(Y^2,N^2,\overline{K}^2)$ and we need to argue further.
	
	\begin{proof}
		The majority of computations that would need to be done here, related to the dynamics of $Y^2$, are very similar to those in the proposition \ref{proposition-smallRBSDE} about small RBSDEs, so we only do the part which is different. 
		
		Define $\overline{f}(s,y,z) = f(s,y+Y^1_s,z+Z^1_s \sigma_s) - f(s,Y^1_s,Z^1_s \sigma_s) + \alpha^2_s$. Note that since $f$ satisfies $(\mathbf{A_{der}})$, $\overline{f}$ satisfies $(\mathbf{A_{locLip}})$ with parameters $(\overline{\beta},\overline{\gamma},2\lambda,r)$, where $\overline{\beta}=f_y(\cdot,Y^1,Z^1\sigma)$ and $\overline{\gamma}=f_z(\cdot,Y^1,Z^1\sigma)$. We have $\norm{\overline{\beta}}_{\infty,1} \le \rho \norm{r}_{\infty,2}^2 < +\infty$ and $\overline{\gamma} \in L^2_{BMO}$.
				
		Following the same approach as for RBSDEs, we first look at the underlying problem of finding $S^2=(Y^2,N^2,K^2)$ solving the perturbation equation (\ref{equation-perturbation.equation.of.reference.notationY2}) when the drift process is
			\begin{align*}
				dV^2_s = dV^2(y^2,n^2)_s= \big[ f(s,y^2_s + Y^1_s,z^2_s \sigma_s + Z^1_s \sigma_s) - f(s,Y^1_s,Z^1_s \sigma_s) + \alpha^2_s \big]dC_s \\
					\qquad + d\Big\langle \nu + \int 2g d(N^1)^\perp,(n^2)^\perp \Big\rangle_s + g_s d\Big\langle (n^2)^\perp \Big\rangle_s \ .
			\end{align*}
		If $S^2$ is a solution, $\overline{S}^2=S^1+S^2$ is then solution to the reflected BSDE (\ref{equation-RBSDE.with.non-circular.driver}) 
		with drift process given by $d\overline{V}^2_s = dV^1(Y^1,N^1)_S+dV^2(y^2,n^2)_s = \big[ f(s,y^2_s + Y^1_s,z^2_s \sigma_s + Z^1_s \sigma_s) + \alpha^2_s \big]dC_s + d\langle \nu, (N^1)^\perp+(n^2)^\perp \rangle_s + g_s d\langle (N^1)^\perp+(n^2)^\perp \rangle_s$. 
		But proposition \ref{proposition-underlyingproblem} guarantees the existence and uniqueness of such an $\overline{S}^2$, hence that of the sought $S^2$. This allows to define a map $Sol'$ from $\mathcal{S}^\infty \times BMO$ to itself. 
		
		Now, to find a solution $S^2$ to the perturbation equation (\ref{equation-perturbation.equation.of.reference.notationY2}), we proceed like in propositions \ref{proposition-smallRBSDE.purelyquadratic} and \ref{proposition-smallRBSDE}, the difference being in dealing with $dK^2$ which is not monotonous anymore here.
		Up to doing the usual transformations (proposition \ref{proposition-smallRBSDE}), let us assume that the drift is purely quadratic as in proposition \ref{proposition-smallRBSDE.purelyquadratic}. Then, Itô's formula first leads to the estimates 
			\begin{align*}
				\abs{\Delta Y^2_t}^2 + \Et{\int_t^T d\bracket{\Delta N^2}_s} \le 2 E \bigg( \int_t^T \Delta Y^2_s d \Delta V^2_s \Big| \F_t \bigg) + 2 E \bigg( \int_t^T \Delta Y^2_s d \Delta K^2_s \Big| \F_t \bigg) , \\
				\abs{Y^2_t}^2 + \Et{\int_t^T d\bracket{N^2}_s} \le \norm{\xi^2}_\infty^2 +  2 E \bigg( \int_t^T Y^2_s d V^2_s \Big| \F_t \bigg) + 2 E \bigg( \int_t^T Y^2_s d K^2_s \Big| \F_t \bigg) \ .
			\end{align*}

		For the term in $\Delta Y^2d\Delta K^2$ one has (even if $L^2 \neq 0$)
			\begin{align*}
				\Delta Y^2d\Delta K^2 &= \big( (Y^{2})' - Y^{2}\big) d(K^{2})' - \big( (Y^{2})' - Y^{2}\big)  dK^{2}				\\
				&= \big((\overline{Y}^{2})' - \overline{Y}^{2}\big) (d(\overline{K}^{2})' - dK^1) - \big((\overline{Y}^{2})' - \overline{Y}^{2}\big) (d\overline{K}^{2} - dK^1)		\\
				&= \big((\overline{Y}^{2})' - \overline{Y}^{2}\big) d(\overline{K}^{2})'  - \big((\overline{Y}^{2})' - \overline{Y}^{2}\big) d\overline{K}^{2} 		\\
				&= \underbrace{(\overline{L}^{2} - \overline{Y}^{2})}_{\le 0} \underbrace{d(\overline{K}^{2})'}_{\ge 0}  - \underbrace{\big( (\overline{Y}^{2})' - \overline{L}^{2} \big)}_{\ge 0}
					\underbrace{d\overline{K}^{2}}_{\ge 0} \le 0 .
			\end{align*}
		For the term in $Y^2dK^2$ one has however, since $L^2=0$, 
			\begin{align*}
				Y^2dK^2 &= (Y^2-L^2)dK^2 + L^2dK^2	\\
				&=  \big((\overline{Y}^2-Y^1)-(\overline{L}^2-L^1)\big)dK^2 + L^2dK^2	\\
				&=  \big((\overline{Y}^2-\overline{L}^2)-(Y^1-L^1)\big)(d\overline{K}^2-dK^1) + L^2dK^2	\\	
				&= \underbrace{(\overline{Y}^2-\overline{L}^2)d\overline{K}^2}_{=0} - \underbrace{(\overline{Y}^2-\overline{L}^2)dK^1}_{\ge 0} 
					- \underbrace{(Y^1-L^1)d\overline{K}^2}_{\ge 0} + \underbrace{(Y^1-L^1)dK^1}_{=0} + L^2dK^2	\\	
				&\le L^2 dK^2 = 0.
			\end{align*}

		Having observed this, the rest is like the analysis of the map $Sol$ and the $\epsilon_0$ is the same.		
		So in the end, provided that that 
			\begin{align*}
				\delta \boldsymbol{D} = \norm{\xi^2}_{\infty} + \norm{\alpha^2}_{\infty,1} \le \epsilon_0(\rho,2\lambda,r)=\frac{e^{-2\rho\abs{r^2}}}{2^{10}(2\lambda)\big( \norm{r}_{\infty,2}^2 + 2 \big)},
			\end{align*}
		there exists a solution $(Y^2,N^2,K^2)$ to the perturbation equation (\ref{equation-perturbation.equation.of.reference.notationY2}).
	\end{proof}

	\begin{remark} \label{remark-uniqueness.for.perturbations.}
		Note that uniqueness holds for the perturbation equations. First, under $(\mathbf{A_{der}})$, $(\mathbf{A_{lLip}})$ holds and so does uniqueness for reflected BSDEs. Then, one can argue that if $Y^2$ and $(Y^{2})'$ are two solutions to (\ref{equation-perturbation.equation.of.reference.notationY2}), then $\overline{Y}^{2}=Y^1+Y^2$ and $(\overline{Y}^{2})'=Y^1+(Y^{2})'$ are two solutions to the same reflected BSDE, so $\overline{Y}^{2}=(\overline{Y}^{2})'$, and therefore $Y^2=(Y^{2})'$. Alternatively we can also argue that if $(Y^2,N^2,K^2)$ is a solution to (\ref{equation-perturbation.equation.of.reference.notationY2}), then $(Y^2,N^2,\overline{K}^2)$ solves a reflected BSDE (\ref{equation-RBSDE.of.reference}) for which uniqueness holds.
	\end{remark}

\subsection{Existence theorem.}

We can now prove the existence theorem of this section.

	\begin{theorem}			\label{proposition-existencetheorem}
		Let $f$ satisfy $(\mathbf{A_{der}})$ with parameters $(\rho, \rho',\lambda,r,h)$ and be such that $f(\cdot,0,0) \in L^{\infty,1}$. Let $\nu \in BMO$, $g \in L^\infty$ be bounded by $\lambda$, $\xi \in L^\infty$, and $L$ be upper bounded. 
		There exists a solution $(Y,N,K) \in \S^\infty \times BMO \times \mathcal{A}$ to the RBSDE (\ref{equation-RBSDE.of.reference}) with data $(f,g,\nu,\xi,L)$. 
	\end{theorem}
	
	\begin{proof}
		The proof is done in two steps. First, we show that one can indeed reduce the problem to the case $L \le 0$, by translation. Existence for the RBSDE with $L \le 0$ is then proved by repeatedly perturbing a solution to a similar RBSDE with smaller data.
		
		\paragraph*{} \emph{Step 1.}
		If $(Y,N,K)$ is a solution to the RBSDE, and $U$ is an upper bound for $L$, set $\overrightarrow{Y}=Y-U$. We see that
			\begin{align*}
				\left\{\begin{aligned}
					&d\overrightarrow{Y} = dY - dU = - dV - dK + dN -0 = -d\overrightarrow{V} - dK + dN						\ ,\\
					&\overrightarrow{Y}_T = \xi - U =: \overrightarrow{\xi}																	\ ,\\
					&\overrightarrow{Y} = Y - U \ge L-U =:\overrightarrow{L}																	\ ,\\
					&1_{\{\overrightarrow{Y}>L-U\}}dK = 1_{\{Y>L\}}dK = 0	\ .
				\end{aligned}\right.
			\end{align*}
		Here we defined
			\begin{align*}
				d\overrightarrow{V} &= dV(Y,N)						\\
				&= dV(\overrightarrow{Y}+U,N)						\\
				&= f(s,\overrightarrow{Y}+U,Z\sigma)  dC + d\langle \nu,N^\perp \rangle + gd\langle N^\perp \rangle	\\
				&= \overrightarrow{f}(s,\overrightarrow{Y},Z\sigma) dC + d\langle \nu,N^\perp \rangle + gd\langle N^\perp \rangle \ .
			\end{align*}		
		It is clear that $\overrightarrow{f}$ still satisfies 	$(\mathbf{A_{der}})$ 	 with parameters $(\rho, \rho',\lambda,r,h)$. And from the assumption on $f_y$ one has
			\begin{align*}
				\abs{\overrightarrow{f}(s,0,0)} = \abs{f(s,U,0)} \le \abs{f(s,0,0)} + \rho r_s^2 U \ ,
			\end{align*}
		so $\overrightarrow{\alpha} = \overrightarrow{f}(\cdot,0,0) \in L^{\infty,1}$.
		
		In the end, $(\overrightarrow{Y},N,K) \in \S^\infty \times BMO \times \mathcal{A}$ is a solution to the reflected BSDE of parameters $(\overrightarrow{f},\nu,g,\overrightarrow{\xi},\overrightarrow{L})$ satisfying the same assumptions, but with $\overrightarrow{L}\le 0$.

		\paragraph*{} \emph{Step 2.}
		We now focus on the case $L \le 0$. Consider $\epsilon_0$ given by proposition \ref{proposition-smallperturbationtoaRBSDE}. 
		For $n \in \N^*$, we define $\xi^{(n)} = \frac{1}{n}\xi$ and $\alpha^{(n)}=\frac{1}{n}\alpha$ (where $\alpha=f(\cdot,0,0)$). 
		We split the data uniformly, that is we consider $\xi^i = \xi^{(n)}$ and $\alpha^i=\alpha^{(n)}$ for all $i \in \{ 1, \ldots, n \}$.
		We choose $n$ big enough so that one has $\boldsymbol{D}^{(n)} := \norm{\xi^{(n)}}_\infty+\norm{\alpha^{(n)}}_{\infty,1} = \frac{1}{n} \boldsymbol{D} \le \epsilon_0$.
		
		First, by proposition \ref{proposition-smallRBSDE}, there exists a solution $(Y^1,N^1,K^1)$ to the RBSDE (\ref{equation-RBSDE.of.reference}) with data $(f-\alpha+\alpha^1,\nu,g,\xi^1,L)$.
		
		Next, for $i = 2$ to $n$, having obtained a solution $(\overline{Y}^{i-1},\overline{N}^{i-1},\overline{K}^{i-1})$ to the RBSDE (\ref{equation-RBSDE.of.reference}) with data $(f-\alpha+\overline{\alpha}^{i-1},\nu,g,\overline{\xi}^{i-1},L)$, proposition \ref{proposition-smallperturbationtoaRBSDE} provides a solution $(Y^i,N^i,K^i)$ to the perturbation equation (\ref{equation-perturbation.equation.of.reference.procedure}) and therefore a solution $(\overline{Y}^{i},\overline{N}^{i},\overline{K}^{i})$ to the RBSDE (\ref{equation-RBSDE.of.reference}) with parameters $(f-\alpha+\overline{\alpha}^i,\nu,g,\overline{\xi}^i,L)$. For $i=n$, since $\overline{\xi}^{n}=\xi$ and $\overline{\alpha}^{n}=\alpha$, $(\overline{Y}^n,\overline{Z}^n,\overline{K}^n)$ is a solution to the RBSDE of interest, which ends the proof.				
	\end{proof}

  
\subsection{Stability in $\S^\infty \times BMO$.}

Given that uniqueness holds, the a posteriori bounds that come with the construction of a perturbation $\delta S = S'-S$ to a solution $S$ in proposition \ref{proposition-smallperturbationtoaRBSDE} readily shows the continuity of the map $(\xi,\alpha) \mapsto (Y,N)$, from $L^\infty \times L^{\infty,1}$ to $\S^\infty \times BMO$.

We now derive an estimate which shows that it is locally Lipchitz, by a sort of bootstrap argument on the above stability result, as well as a BMO-norm equivalence. 
In the proposition below, we consider a fixed set of data $(f,\nu,g,\xi,L)$ and the associated solution $S=(Y,N,K)$, and we define $\alpha = f(\cdot,0,0)$. 
Now, for close data $(f+\delta\alpha',\nu,g,\xi',L)$ and $(f+\delta\alpha'',\nu,g,\xi'',L)$, we consider the solutions $S'$ and $S''$. Set $\delta \xi' = \xi'-\xi$ and $\delta \xi'' = \xi''-\xi$.
We use the notation $\delta S' = S'-S$, $\delta S'' = S''-S$ for the perturbations around $S$ and $\Delta S = S''-S' = \delta S'' - \delta S'$. 
What we show is that if $(\delta \xi',\delta\alpha')$ and $(\delta \xi'',\delta\alpha'')$ are sufficientily small, 
the distance $\norm{\Delta S}$ is linearly controlled by the distance 
$\Delta \boldsymbol{D} = \norm{\Delta \xi}_\infty + \norm{\Delta\alpha}_{\infty,1} = \norm{\xi''-\xi'}_\infty + \norm{\delta\alpha''-\delta\alpha'}_{\infty,1}$. 
That is, $(\xi',\alpha') \mapsto (Y',N')$ is locally Lipschitz at the point $(\xi,\alpha)$.

	\begin{proposition}			\label{proposition-stability.estimate} 
		Suppose that $f$ satisfies $(\mathbf{A_{der}})$ with parameters $(\rho, \rho',\lambda,r,h)$, that $\alpha=f(\cdot,0,0) \in L^{\infty,1}$, that $\nu \in BMO$, that $g$ is bounded by $\lambda$ and that $L$ is upper bounded. We consider $\xi \in L^\infty$ and the solution $(Y,N,K)$ to the reflected BSDE of parameters $(f,\nu,g,\xi,L)$. 
		
		Now, for any $(\xi',\delta \alpha')$ and $(\xi'',\delta \alpha'') \in L^\infty \times L^{\infty,1}$, let $S'=(Y',N',K')$ and $S''=(Y'',N'',K'')$ be the solutions to the reflected BSDEs of parameters $(f+\delta \alpha',\nu,g,\xi',L)$ and $(f+\delta \alpha'',\nu,g,\xi'',L)$ respectively. 
		
		If $\delta \boldsymbol{D}' = \norm{\delta \xi^{'}}_\infty + \norm{\delta \alpha^{'}}_{\infty,1}$ and $\delta \boldsymbol{D}'' = \norm{\delta \xi^{''}}_\infty + \norm{\delta \alpha^{''}}	_{\infty,1}$
		satisfy
			\begin{align*}
				\delta \boldsymbol{D}'  \text{ and } \delta \boldsymbol{D}'' 
							\le \frac{1}{\sqrt{2}} \epsilon_0(\overline{\beta},2\lambda,r) = \frac{1}{\sqrt{2}} \frac{e^{-2\norm{\overline{\beta}}_{\infty,1} }}{2^{10} (2\lambda) \big( \norm{r}_{\infty,2}^2 + 2 \big)} \ ,
			\end{align*} 
		where $\overline{\beta}=f_y(\cdot,Y,Z\sigma)$, then we have
			\begin{align*}
				&\norm{Y''-Y'}_{\S^\infty} \le 2^5 e^{2\norm{\overline{\beta}}_{\infty,1}}\big( \norm{\xi''-\xi'}_\infty + \norm{\alpha''-\alpha'}_{\infty,1} \big)		\qquad\text{and} \\
				&\norm{N''-N'}_{BMO(P)} \le 2^5 C(Y,N)\, e^{2\norm{\overline{\beta}}_{\infty,1}}\big( \norm{\xi'' - \xi'}_{\infty} + \norm{\alpha''-\alpha'}_{\infty,1} \big) \ ,
			\end{align*}
		where $C(Y,N)$ is a constant depending on $(Y,N)$. 
	\end{proposition}

	\begin{proof}
		We know that $\overline{f}(s,\delta y, \delta z) := f(s,Y_s+\delta y,Z_s \sigma_s + \delta z) - f(s,Y_s,Z_s \sigma_s)$ satisfies $(\mathbf{A_{locLip}})$ 
		with parameters $(\overline{\beta},\overline{\gamma},2\lambda,r)$, where $\overline{\beta}=f_y(\cdot,Y,Z\sigma)$ and $\overline{\gamma}=f_z(\cdot,Y,Z\sigma)$ satisfy 
		$\norm{\overline{\beta}}_{\infty,1} \le \rho\norm{r}_{\infty,2}^2$ and $\overline{\gamma} \in L^2_{BMO}$ ; and $\overline{\nu}=\nu + \int 2gdN \in BMO$. We linearize $\overline{f}$ like in proposition \ref{proposition-smallRBSDE} : $\overline{f}(s,\delta y,\delta z) = \overline{\beta}_s \delta y + \overline{\gamma}_s \delta z+ \overline{h}(s,\delta y,\delta z)$.

		Since the difference $\Delta Y = Y'' - Y'$ has the dynamics
			\begin{align*}
				d\Delta Y_s &= - \big[ \Delta \alpha_s + \overline{\beta}_s \Delta Y_s + \overline{\gamma}_s \Delta Z_s\sigma_s 
														+ \{ \overline{h}(s,\delta Y''_s,\delta Z''_s\sigma_s) - \overline{h}(s,\delta Y',\delta Z'_s \sigma_s) \} \big] dC_s			\\
					&\qquad - d\bracket{\overline{\nu},(\Delta N)^\perp}_s - g_s d\bracket{ (\delta N'')^{\perp}+(\delta N')^{\perp} , (\Delta N)^\perp }_s - d\Delta K_s + d\Delta N_s	 \ ,			
			\end{align*}
		doing the usual transformations, with $\frac{dQ}{dP} = \Exp(\int \overline{\gamma}\sigma^{-1}dM + \overline{\nu} )$ and $\overline{B} = e^{\int_0^\cdot \overline{\beta}_u dC_u}$, 
		the standard computations give, like in (\ref{equation-small.fixedpoint.DeltaSol.final.estimate}),
			\begin{align*}
				\norm{\widehat{\Delta S}}^2_Q 
					\le 2^9 \widehat{\Delta E} + 2^9 \widehat{2\lambda}^2(\norm{r}_{\infty,2}^2+2)^2 \big( \norm{\widehat{\delta S''}}^2_Q+\norm{\widehat{\delta S'}}^2_Q \big)  \norm{\widehat{\Delta S}}^2_Q
			\end{align*}
			
		But $\delta S''$ and $\delta S'$ are the unique solutions to the perturbation equations (\ref{equation-perturbation.equation.of.reference.notationY2}) 
		with data $(\overline{f}+\delta \alpha'',\overline{\nu},g,\delta \xi'',L)$ and $(\overline{f}+\delta \alpha',\overline{\nu},g,\delta \xi',L)$, and by the way they were constructed in proposition \ref{proposition-smallperturbationtoaRBSDE} 
		(recall that $\delta \boldsymbol{D}', \delta \boldsymbol{D}''  \le \exp(-\norm{\overline{\beta}}_{\infty,1}) \frac{1}{2^5} \frac{R_0(\widehat{2\lambda},r)}{\sqrt{2}}$) we know that they satisfy
			\begin{align*}
				\norm{\widehat{\delta S''}}^2_Q,\ \norm{\widehat{\delta S'}}^2_Q \le \frac{ R_0(\widehat{2\lambda},r)^2 }{2} \ ,			\qquad \text{ so}\\
				\norm{\widehat{\delta S''}}^2_Q + \norm{\widehat{\delta S'}}^2_Q \le R_0(\widehat{2\lambda},r)^2 = \frac{1}{ 2^{10} \widehat{2\lambda}^2 (\norm{r}^2+2)^2 } \ .
			\end{align*}
		Reinjecting this in the previous estimate we have 
				$\norm{\widehat{\Delta S}}^2_Q \le 2^9 \widehat{\Delta E} + \half \norm{\widehat{\Delta S}}^2_Q$
		and therefore
			\begin{align*}
				\norm{\widehat{\Delta S}}^2_Q = \norm{\widehat{\Delta Y}}^2_{\S^\infty} + \norm{\widehat{\widetilde{\Delta N}}}^2_{BMO(Q)} \le 2^{10}\widehat{\Delta E} \ .
			\end{align*}
		Then this implies that $\norm{\widehat{\Delta Y}}_{\S^\infty} \le 2^{5} \widehat{\Delta \boldsymbol{D}}$ 
		and so $\norm{\Delta Y}_{\S^\infty} \le 2^{5} e^{2\norm{\overline{\beta}}_{\infty,1} } \Delta \boldsymbol{D}$. 
		For the same reason, $\norm{\widetilde{\Delta N}}_{BMO(Q)} \le 2^{5} e^{2\norm{\overline{\beta}}_{\infty,1}} \Delta \boldsymbol{D}$. 
		By theorem 3.6 in Kazamaki, $\norm{\Delta N}_{BMO(P)} \le C(Q) \norm{\widetilde{\Delta N}}_{BMO(Q)}$ where the constant depends only on $Q$, or equivalently on the martingale $\int \overline{\gamma}\sigma^{-1}dM + \overline{\nu}$, and \textit{in fine} on $(Y,N)$.
	\end{proof}

Note that the interesting part of the above result is the martingale estimate. Indeed, the estimate for $Y''-Y'$ in $\S^\infty$ actually holds for any size of data (as can be seen by linearizing the drift, doing a change of measure to get rid of all the terms in $N$ and solving for $Y$). 
As mentionned in the introduction, we know that $(\xi,\alpha) \mapsto N$ is global Lipschitz in $\H^p$, and $\half$-Hölder in $BMO$. The above estimate shows it is in fact locally Lipschitz in $BMO$.

  
\section{Existence under more general assumptions.}					\label{section-generalisation}

In theorem \ref{proposition-existencetheorem}, the existence of a solution was proved under $(\mathbf{A_{der}})$, so in particular under the assumption that $f$ is a Lipschitz function of $y$, and therefore at most linear in $y$. 
In this section, we extend this result to more general assumptions on $f$.

To some extent, we would like to replace $\rho,\rho',\lambda$ which are constants in $(\mathbf{A_{der}})$ by arbitrary growth functions (while of course still assuming that $f$ ends up with a growth in $y$ compatible with existence of solutions). 
Looking back at proposition \ref{proposition-smallperturbationtoaRBSDE}, we see that when $\rho$ is a growth function, the maximal size $\epsilon$ allowed for a perturbation $(\xi^2,\alpha^2)$ of the parameters would depend on the size $\norm{Y^1}_{\S^\infty}$ of the solution. 
It is therefore not clear that one can choose $\epsilon_0$ and the decomposition $\xi=\sum_{i=1}^n \xi^i$, $\alpha = \sum_{i=1}^n \alpha^i$ uniformly for the perturbation procedure in the proof of theorem \ref{proposition-existencetheorem}, or to put things differently, that a series of pertubations could terminate in finitely many steps. 
This however can be guaranteed if one can obtain an \emph{a priori} bound for the solutions to reflected BSDEs with drift $(f,\nu,g)$.

\paragraph{Case of a superlinear growth in $y$.}\ \\
In the following theorem, we extend theorem \ref{proposition-existencetheorem} to the case where $f$ can have slightly-superlinear growth in $y$.

	\begin{theorem}
		Consider a set of data $(f,\nu,g,\xi,L)$ satisfying the assumptions of theorem \ref{proposition-existencetheorem}, but with $\rho, \rho',\lambda$ in $(\mathbf{A_{der}})$ being growth function instead of constants.
		Further assume that $\abs{f(t,y,0)} \le \abs{f(t,0,0)} + \varphi(y)$ for a growth function $\varphi$ such that $\int_1^{+\infty} \frac{1}{\varphi(y)}dy = +\infty$.
		Then there exists a solution $(Y,N,K)$ to the reflected BSDE (\ref{equation-RBSDE.of.reference}) with data $(f,\nu,g,\xi,L)$. 
	\end{theorem}

	\begin{proof}
		We will apply the perturbation procedure as was done previously when $\rho,\rho',\lambda$ were constants.
		
		First, by the estimate in theorem 1 in Kobylanski \emph{et al.} \cite{KobylanskiAndLepeltierAndQuenezAndTorres}, we know that there exists a function $F$ increasing (a growth function) such that for any set of data $(f,\nu,g,\xi,L)$ satisfying the assumptions and for any solution $(Y,N,K)$ we have $\norm{Y}_{\S^\infty} \le F(\norm{\xi}_\infty,\norm{\alpha}_{\infty,1})$. 
		Now, for a fixed set of data, we define $\rho_{\text{max}} = \rho\big(F(\norm{\xi}_\infty,\norm{\alpha}_{\infty,1})\big)$. 
		We fix $n$ big enough that
			\begin{align*}
				\boldsymbol{D}^{(n)}=\frac{\boldsymbol{D}}{n} 
										\le \frac{e^{-2 \rho_{\text{max}}\norm{r}_{\infty,2}^2}}{2^{10} \big(2\lambda(1)\big) \big( \norm{r}_{\infty,2}^2+2 \big)} = \epsilon_0(\rho_{\text{max}}, 2\lambda(1),r) \ .
			\end{align*}
		We will construct $n$ solutions $S^i$ of reflected BSDEs or perturbation equations such that for each equation, the size of the data is $\boldsymbol{D^i}=\boldsymbol{D^{(n)}}$ and the size of the solution is such that $\norm{\widehat{Y^i}}_{S^\infty} \le 1$. Note that the $\widehat{\ \ }$ here indicates the multiplication by $B^i=\exp\big(\int_0^\cdot f_y(\overline{S}^{i-1}_u) dC_u \big)$.
		
		We know that we can do a transation to be reduced to the case $L \le 0$ so we assume from now on that $L \le 0$. 
		Define, for $i=1 \ldots n$, $\xi^i := \xi^{(n)} = \frac{1}{n} \xi$ and $\alpha^i := \alpha^{(n)} = \frac{1}{n} \alpha$ (uniform decomposition of $\xi$ and $\alpha$). 
		
		\paragraph*{}
		For $i=1$, we first build a solution $S^1=(Y^1,N^1,K^1)$ to the reflected BSDE (\ref{equation-RBSDE.of.reference}) with parameters $(f-\alpha+\alpha^1,\nu,g,\xi^1,L)$. 
		Proposition \ref{proposition-smallRBSDE} as it is stated doesn't strictly apply, but we can adapt the proof. 
		We define the integrating factor $B=e^{\int \beta dC}$ with $\beta =\overline{\beta}^{0} = f_y(\cdot,0,0) \in L^{\infty,1}$ 
		and the new measure $Q$ by $\frac{dQ}{dP} = \Exp(\int \gamma \sigma^{-1}dM + \nu)$ where $\gamma = \overline{\gamma}^{0} = f_z(\cdot,0,0) \in L^2_{BMO}$. 
		Then, like in proposition \ref{proposition-smallRBSDE.purelyquadratic}, we look for a solution $(\widehat{Y^1},\widehat{\widetilde{N^1}},\widehat{K^1})$ 
		to the reflected BSDE with no linear term, via the fixed point theorem. 
		We look for a solution in a ball of radius $R$ and now further demand that $R\le 1$, so that the conditions to be met are that
			\begin{align*}
				R \le R_0(\widehat{2\lambda(1)},r) = \frac{1}{2^5 \widehat{2\lambda(1)} \big(  \norm{r}_{\infty,2}^2 + 2 \big)}		\text{ and }
				\widehat{\boldsymbol{D}^1} = \frac{\widehat{\boldsymbol{D}}}{n} \le \frac{R_0(\widehat{2\lambda(1)},r)}{2^5} = \epsilon_0( \widehat{2\lambda(1)},r) \ ,
			\end{align*}
		where $\widehat{2\lambda(1)} = e^{\norm{\overline{\beta}^0}_{\infty,1}} (2\lambda(1))$. 
		Now, since we have chosen $n$ such that $\frac{\boldsymbol{D}}{n} \le \frac{e^{-2 \rho_{\text{max}}\norm{r}_{\infty,2}^2}}{2^{10} (2\lambda(1)) \big(\norm{r}_{\infty,2}^2+2 \big)}$, 
			\begin{align*}
				\frac{\widehat{\boldsymbol{D}}}{n} 
					\le e^{\norm{\overline{\beta}^0}_{\infty,1}} \frac{\boldsymbol{D}}{n} 
					\le  \frac{ e^{\norm{\overline{\beta}^0}_{\infty,1}} e^{-2 \rho_{\text{max}}\norm{r}_{\infty,2}^2} }{2^{10}  \big( 2\lambda(1) \big) \big(\norm{r}_{\infty,2}^2 + 2 \big)}
					&=  \frac{ e^{2\norm{\overline{\beta}^0}_{\infty,1}} e^{-2 \rho_{\text{max}}\norm{r}_{\infty,2}^2} }{2^{10} \widehat{2\lambda(1)} \big(\norm{r}_{\infty,2}^2 + 2 \big)} \\
					&\le \frac{1}{2^{10} \widehat{2\lambda(1)} \big(\norm{r}_{\infty,2}^2 + 2 \big)} 
						= \epsilon_0( \widehat{2\lambda(1)},r)
			\end{align*}
		because $\norm{\overline{\beta}_{\infty,1}^0} \le \rho(0)\norm{r}_{\infty,2}^2$ and $\rho(0) \le \rho_{\text{max}}$ by construction. So we indeed get a solution $(\widehat{Y^1},\widehat{\widetilde{N^1}},\widehat{K^1})$ and doing the reverse transforms gives a solution $S^1=(Y^1,N^1,K^1)$ to the reflected BSDE with the linear terms.
		
		\paragraph*{}
		For $i=2 \ldots n$, we have a solution $\overline{S}^{i-1}$ to the reflected BSDE (\ref{equation-RBSDE.of.reference}) with parameters $(f-\alpha+\overline{\alpha}^{i-1},\nu,g,\overline{\xi}^{i-1},L)$ and want to construct the appropriate perturbation $S^i$. We simply do the same computations as in proposition \ref{proposition-smallperturbationtoaRBSDE}, using the integrating factor $\overline{B}^{i-1}=e^{\int \overline{\beta}^{i-1} dC}$ where $\overline{\beta}^{i-1} = f_y(\cdot,\overline{Y}^{i-1},\overline{Z}^{i-1}\sigma)$, and the change of measure $\frac{dQ}{dP}=\Exp( \int \overline{\gamma}^{i-1}\sigma^{-1}dC + \overline{\nu}^{i-1} )$ where $\overline{\gamma}^{i-1}=f_z(\cdot,\overline{Y}^{i-1},\overline{Z}^{i-1}\sigma)$ and $\overline{\nu}^{i-1} = \nu + \int 2 g d(\overline{N}^{i-1})^\perp$. 

		Because we know, by the a priori estimate on solutions of the reflected BSDE, that 
			\begin{align*}
				\norm{\overline{Y}^{i-1}}_{\S^\infty} \le 
									F\left(\Norm{\frac{(i-1)}{n}\,\xi}_\infty,\Norm{\frac{(i-1)}{n}\,\alpha}_{\infty,1}\right) \le F\big(\norm{\xi}_\infty,\norm{\alpha}_{\infty,1}\big),
			\end{align*} 
		we know that $\norm{\overline{\beta}^{i-1}}_{\infty,1} \le \rho(\norm{\overline{Y}^{i-1}}_{\S^\infty}) \norm{r}_{\infty,2}^2 \le \rho_{\text{max}} \norm{r}_{\infty,2}^2$. 
		Therefore, just as above, we find that the size $\widehat{\boldsymbol{D}^i}$ of the data is indeed small enough, and so we can construct the perturbation $S^i$.
	\end{proof}	

As can be seen from the proof above, the key to the generalization is to have an a priori estimate $\norm{Y}_{\S^\infty} \le F(\norm{\xi}_\infty,\norm{\alpha}_{\infty,1})$ for some growth function $F$.

\paragraph{Case of $f$ monotone and with arbitraty growth in $y$.}\ \\
We can also generalize the result of theorem \ref{proposition-existencetheorem} to the case where $f$ is so-called monotonous (or 1-sided Lipschitz) in $y$, with arbitrary growth.

	\begin{theorem}
		Consider a set of parameters $(f,\nu,g,\xi,L)$ satisfying the assumptions of theorem \ref{proposition-existencetheorem}, but with $\rho, \rho',\lambda$ in $(\mathbf{A_{der}})$ being growth function instead of constants.
				
		Further assume that $\abs{f(t,y,0)} \le \abs{f(t,0,0)} + \varphi(y)$ for a growth function $\varphi$ and that there exists a constant $\mu$ such that for all $y,y',z,s,\omega$,
			\begin{align*}
				(y'-y)\big( f(s,y',z) - f(s,y,z) \big) \le \mu \, r_s^2 \, \abs{y'-y}^2
			\end{align*}

		Then there exists a solution $(Y,N,K)$ to the reflected BSDE (\ref{equation-RBSDE.of.reference}) with parameters $(f,\nu,g,\xi,L)$. 
	\end{theorem}

As remarked above, it is enough to have an a priori estimate for $\norm{Y}_{\S^\infty}$. One can use the one obtained in the proof of theorem 3.1 in \cite{Xu}. Alternatively, having argued that it is enough to study the case where the obstacle is negative, one can linearize the driver in the $N$ variable, and do a measure change. Then, using Itô with $\abs{\,\cdot\, }^2$ to take advantage of the monotonicity condition, one could conclude via standard estimations that 
	\begin{align*}
		\norm{Y}_{\S^\infty}^2 \le 2 e^{4\mu \norm{r}_{\infty,2}^2} \big( \norm{\xi}_\infty^2 + 2 \norm{\alpha}_{\infty,1}^2) =: F\big( \norm{\xi}_\infty,\norm{\alpha}_{\infty,1} \big)^2
	\end{align*}

\section{Conclusion.}

We obtained the standard well-posedness results (comparison, uniqueness, existence) for a general class of quadratic reflected BSDEs driven by a continous martingale. 

We also proved under minimal assumptions the special comparison theorem (which allows to compare the increasing processes of two solutions), using a new proof which does not rely on comparison for BSDEs (which requires regularity assumptions on $f$) but is more in the spirit of the Snell-envelopes view of reflected BSDEs.

Finally, we also showed a local Lipschitz estimate in $BMO$ for the martingale part of the solution, which improves on the previously known regularity. The idea is somehow to bootstrap an existing, weaker regularity result.

\paragraph*{}
For the existence of a solution, we first worked under the standard assumption that $f$ is Lipschitz in $y$ and adapted the technique introducted by Tevzadze in \cite{Tevzadze}. 
The perturbation procedure required a special attention here since perturbations to a reflected BSDE don't satisfy a reflected BSDE. This is linked to the fact that the underlying problem for reflected BSDEs is a Snell envelope problem, and that if $L$ and $L'$ are obstacles, the Snell envelope of $L+L'$ is not the sum of that of $L$ and that of $L'$.
The problems with the difference of increasing processes not being increasing can be avoided if one does not perturb the obstacle during the procedure. We therefore applied first a transformation to the reflected BSDE to be led to study only the case where $L \le 0$. We expect that this approach would also work for doubly reflected BSDEs if they can be reduced to studying the case where the lower obstacle $L$ is negative and the upper obstacle $U$ is positive.

However, this technique would probably not work in the case of an unbounded terminal condition. In that setting indeed, the martingale part of the solution is not in $BMO$, a property which is well put to contribution here. 

We also showed the existence of solutions in the case where $f$ has slightly-superlinear growth in $y$ and in the case where $f$ is monotone with arbitrary growth in $y$. For this, and unlike in the case where $f$ is a Lipschitz function of $y$, we need to use an a priori bound to guarantee that the pertubation procedures ends in finitely many perturbations (and can be carried uniformly). Therefore, the same technique can be used to construct solutions in these three cases.

\paragraph{Acknowledgements.} We would like to thank the reviewers for their helpful remarks and comments which help improve the first version of this manuscript. We also thank Sam Cohen for helpful comments and discussions. \\
This research was supported by the Engineering and Physical Sciences Research Council (UK) via grant EP/P505216/1 and by the Oxford-Man Institute.

\paragraph*{}
\ \\ 
{\small Arnaud Lionnet, \\ 18th October 2013 \ (version 2). \\ Oxford-Man Institute, Eagle House, OX2 6ED, Oxford, United Kingdom. } \\

\bibliographystyle{plain}

\begin{thebibliography}{10}

\bibitem{AnkirchnerAndImkellerAndDosReis}
Ankirchner, Imkeller, and Dos Reis.
\newblock Classical and variational differentiability of {BSDE}s with quadratic
  growth.
\newblock {\em Electronic Journal of Probability}, 12(53):1418--1453, 2007.

\bibitem{BarrieuAndElKaroui}
Barrieu and El~Karoui.
\newblock Monotone stability of quadratic semimartingales with applications to
  unbounded general quadratic {BSDE}s.
\newblock {\em Annals of Probability}, 41:1831--1853, 2013.

\bibitem{BayraktarAndYao}
Bayraktar and Yao.
\newblock Quadratic reflected {BSDE}s with unbounded obstacles.
\newblock {\em Stochastic processes and their applications}, 122(4):1155--1203,
  2012.

\bibitem{BriandAndConfortola}
Briand and Confortola.
\newblock {BSDE}s with stochastic lipschitz condition and quadratic {PDE}s in
  hilbert spaces.
\newblock {\em Stochatic Processes and their Applications}, 118(5):818--838,
  2008.

\bibitem{BriandAndDelyonAndHuAndPardouxAndStoica}
Briand, Delyon, Hu, Pardoux, and Stoica.
\newblock {$L^p$} solutions of backward stochastic differential equations.
\newblock {\em Stochatic Processes and their Applications}, 108(1):109--129,
  2003.

\bibitem{BriandAndElie}
Briand and Elie.
\newblock A simple constructive approach to quadratic {BSDE}s with or without
  delay.
\newblock {\em Stochastic processes and their applications}, 123(8):2921--2939,
  2013.

\bibitem{BriandAndHu2006}
Briand and Hu.
\newblock {BSDE}s with quadratic growth and unbounded terminal value.
\newblock {\em Probability Theory and Related Fields}, 136(4):604--618, 2006.

\bibitem{BriandAndHu2008}
Briand and Hu.
\newblock Quadratic {BSDE}s with convex generators and unbounded terminal
  conditions.
\newblock {\em Probability Theory and Related Fields}, 141(3--4):543--567,
  2008.

\bibitem{BriandAndLepeltierAndSanMartin}
Briand, Lepeltier, and San Martin.
\newblock One-dimensional backward stochastic differential equations whose
  coefficient is monotonic in y and non-lipschitz in z.
\newblock {\em Bernoulli}, 13(1):80--91, 2007.

\bibitem{CohenAndElliott}
Cohen and Elliott.
\newblock Existence uniqueness and comparison for {BSDE}s in general spaces.
\newblock {\em Annals of Probability}, 40(5):2264--2297, 2012.

\bibitem{DelbaenAndHuAndRichou2011}
Delbaen, Hu, and Richou.
\newblock On the uniqueness of solutions to quadratic {BSDE}s with convex
  generators and unbounded terminal conditions.
\newblock {\em Annales de l'Institut Henri Poincaré, Probability and
  Statistics}, 47(2):559--574, 2011.

\bibitem{DelbaenAndHuAndRichou2013}
Delbaen, Hu, and Richou.
\newblock On the uniqueness of solutions to quadratic {BSDE}s with convex
  generators and unbounded terminal conditions : the critical case.
\newblock {\em arXiv:1303.4859}, 2013.

\bibitem{HamadeneAndLepeltierAndMatoussi}
Hamadène, Lepeltier, and Matoussi.
\newblock Double barrier reflected backward {SDE}'s with continuous
  coefficient.
\newblock In El~Karoui and Mazliak, editors, {\em Backward stochastic
  differential equations}, pages 161--175, 1997.

\bibitem{HuAndImkellerAndMuller}
Hu, Imkeller, and Müller.
\newblock Utility maximization in incomplete markets.
\newblock {\em Annals Of Applied Probability}, 15(3):1691--1712, 2005.

\bibitem{ElKarouiAndHuang}
El~Karoui and Huang.
\newblock A general result of existence and uniqueness of backward stochastic
  differential equations.
\newblock In El~Karoui and Mazliak, editors, {\em Backward stochastic
  differential equations}, pages 27--36, 1997.

\bibitem{EKnKnPnPnQ}
El~Karoui, Kapoudjian, Pardoux, Peng, and Quenez.
\newblock Reflected solutions of backward {SDE}'s and related obstacle problems
  for {PDE}'s.
\newblock {\em Annals of Probability}, 25(2):702--737, 1997.

\bibitem{ElKarouiAndPengAndQuenez}
El~Karoui, Peng, and Quenez.
\newblock Backward stochastic differential equations in finance.
\newblock {\em Mathematical Finance}, 7(1):1--71, 1997.

\bibitem{Kazamaki}
Kazamaki.
\newblock {\em Continuous exponential martingales and BMO}.
\newblock Springer, 1994.

\bibitem{KaziTaniAndPossamaiAndZhou}
Kazi-Tani, Possamai, and Zhou.
\newblock Quadratic {BSDE}s with jumps and related non-linear expectations: a
  fixed-point approach.
\newblock {\em arXiv:1208.5581}, 2012.

\bibitem{Kobylanski}
Kobylanski.
\newblock Backward stochastic differential equations and partial differential
  equations with quadratic growth.
\newblock {\em Annals of probability}, 28(2):558--602, 2000.

\bibitem{KobylanskiAndLepeltierAndQuenezAndTorres}
Kobylanski, Lepeltier, Quenez, and Torrès.
\newblock Reflected {BSDE}s with superlinear quadratic coefficient.
\newblock {\em Probability and Mathematical Statistics}, 22(1):51--83, 2002.

\bibitem{LepeltierandSanMartin98}
Lepeltier and San Martin.
\newblock Existence for {BSDE}s with superlinear-quadratic coefficient.
\newblock {\em Stochastics and Stochastic Reports}, 63(3--4):227--240, 1998.

\bibitem{LepeltierAndMatoussiAndXu}
Lepeltier, Matoussi, and M.~Xu.
\newblock Reflected backward stochastic differential equations under
  monotonicity and general increasing growth conditions.
\newblock {\em Adv. in Appl. Probab.}, 37(1):134--159, 2005.

\bibitem{LepeltierAndXu}
Lepeltier and M.~Xu.
\newblock Reflected {BSDE}s with quadratic growth and unbounded terminal value.
\newblock {\em arXiv:0711.0619v1}, 2007.

\bibitem{Morlais}
Morlais.
\newblock Quadratic {BSDE}s driven by continuous martingales and applications
  to the utility maximization problem.
\newblock {\em Finance and Stochastics}, 13(1):121--150, 2009.

\bibitem{Pardoux}
Pardoux.
\newblock {BSDE}s, weak convergence and homogenization of semilinear {PDE}s.
\newblock In {\em Nonlinear analysis, differential equations and control
  (Montreal, QC, 1998)}, pages 503--549, 1999.

\bibitem{PardouxAndPeng}
Pardoux and Peng.
\newblock Adapted solution of a backward stochastic differential equation.
\newblock {\em Systems \& Control Letters}, 14(1):55--61, 1990.

\bibitem{PardouxAndZhang}
Pardoux and Zhang.
\newblock Generalized {BSDE}s and nonlinear neumann boundary value problems.
\newblock {\em Probability Theory And Related Fields}, 110(4):535--558, 1998.

\bibitem{PengAndXu}
S.~Peng and M.~Xu.
\newblock The smallest g-supermartingale and reflected {BSDE}s with single and
  double {$L^2$} obstacles.
\newblock {\em Annales de l'Institut Henri Poincaré, Probability and
  Statistics}, 41(3):605--630, 2005.

\bibitem{Tevzadze}
Tevzadze.
\newblock Solvability of backward stochstic differential equations with
  quadratic growth.
\newblock {\em Stochastic processes and their applications}, 118(3):503--515,
  2008.

\bibitem{Touzi}
Touzi.
\newblock {\em Optimal Stochastic Control, Stochastic Target Problems, and
  Backward {SDE}s}.
\newblock Fields Institute Monograph, Springer, 2013.

\bibitem{Xu}
M.~Xu.
\newblock Backward stochastic differential equations with reflection and weak
  assumptions on the coefficients.
\newblock {\em Stochastic processes and their applications}, 118(6):968--980,
  2008.

\end{thebibliography}

\end{document}